\documentclass{amsart}
\usepackage{amsmath,amsthm}
\usepackage{amsfonts,amssymb}
\usepackage{accents}
\usepackage{enumerate}
\usepackage{accents,color}
\usepackage{graphicx}
\usepackage{pict2e}
\newcommand{\lvt}{\left|\kern-1.35pt\left|\kern-1.3pt\left|}
\newcommand{\rvt}{\right|\kern-1.3pt\right|\kern-1.35pt\right|}

\newtheorem{thm}{Theorem}[section]

\newtheorem{lem}[thm]{Lemma}
\newtheorem{prop}[thm]{Proposition}
\newtheorem{exam}[thm]{Example}

\newtheorem{defn}[thm]{Definition}

\theoremstyle{remark}
\newtheorem{rem}{Remark}[section]

 \def\la{{\langle}}
 \def\ra{{\rangle}}

 \def\d{\mathrm{d}}
 
 \def\sph{{\mathbb{S}^{d-1}}}

 \def\sb{{\mathsf b}}

 \def\sh{{\mathsf h}}
  
 \def\sw{{\mathsf w}}

 \def\sE{{\mathsf E}}
 \def\sG{{\mathsf G}}
 
 \def\sK{{\mathsf K}}
 
 \def\sO{{\mathsf O}}
 \def\sP{{\mathsf P}}
 \def\sQ{{\mathsf Q}}

 \def\sT{{\mathsf T}}
 
 \def\sW{{\mathsf W}}

 \def\fa{{\mathfrak a}}
 \def\fb{{\mathfrak b}}
 \def\fc{{\mathfrak c}}

  \def\ft{{\mathfrak t}}
 \def\fD{{\mathfrak D}}

 \def\a{{\alpha}}
 \def\b{{\beta}}
 \def\g{{\gamma}}
 \def\k{{\kappa}}
 
 \def\l{{\lambda}}
 \def\o{{\omega}}
 \def\s{\sigma}
 \def\la{{\langle}}
 \def\ra{{\rangle}}

 \def\tb{{\mathbf t}}
 \def\ub{{\mathbf u}}
 \def\vb{{\mathbf v}}
 \def\xb{{\mathbf x}}
 \def\yb{{\mathbf y}}

 \def\Kb{{\mathbf K}}
 \def\Pb{{\mathbf P}}
 \def\Qb{{\mathbf Q}}

 \def\Wb{{\mathbf W}}
 
 \def\CD{{\mathcal D}}
 
 \def\CH{{\mathcal H}}

 \def\CV{{\mathcal V}}

 \def\BB{{\mathbb B}}

 \def\NN{{\mathbb N}}

 \def\RR{{\mathbb R}}

 \def\VV{{\mathbb V}}
 
 \def\XX{{\mathbb X}}
 \def\ZZ{{\mathbb Z}}
      \def\proj{\operatorname{proj}}

\def\lla{\langle{\kern-2.5pt}\langle}      
\def\rra{\rangle{\kern-2.5pt}\rangle}

\newcommand{\wh}{\widehat}

\def\f{\frac}

\graphicspath{{figures/}}
\begin{document}
 
\title{Approximation and orthogonality on fully symmetric domains}
\author{Yuan~Xu}
\address{Department of Mathematics, University of Oregon, Eugene,
OR 97403--1222, USA}
\email{yuan@uoregon.edu}
\thanks{The author was partially supported by Simons Foundation Grant \#849676}
\date{\today}
\subjclass[2010]{33C45, 42C05, 42C10, 65D15, 65D20.}
\keywords{Orthogonal polynomials, domains of revolution, spectral operator, addition formula.}

\begin{abstract}
We study orthogonal polynomials on a fully symmetric planar domain $\Omega$ that is generated by 
a certain triangle in the first quadrant. For a family of weight functions on $\Omega$, we show that 
orthogonal polynomials that are even in the second variable on $\Omega$ can be identified with 
orthogonal polynomials on the unit disk composed with a quadratic map, and the same phenomenon 
can be extended to the domain generated by the rotation of $\Omega$ in higher dimensions. The  
connection allows an immediate deduction of results for approximation and Fourier orthogonal 
expansions on these fully symmetric domains. It applies, for example, to analysis on a double cone 
or a double hyperboloid. 
\end{abstract}

\maketitle

\section{Introduction}
\setcounter{equation}{0}

Most of the study of orthogonal polynomials (OPs) in several variables is carried out on structured domains, and 
one of the most well-studied cases is the unit ball (see, for example, \cite{ACH, DaiX1, DaiX, DX, X05}). 
OPs on the unit ball share, together with a handful of other families on regular domains, two characteristic 
properties of classical OPs: they are eigenfunctions of a second-order differential operator, and their 
reproducing kernel satisfies a closed-form formula, both of which serve as essential tools for studying 
approximation and Fourier orthogonal expansions (cf. \cite{DaiX1, DaiX, Dit, DX, PX, W} and their references). 
They are also building blocks for OPs on domains of revolution that have been studied recently 
\cite{OX, X21a, X21, X23a, X24} and reference therein. In particular, orthogonal structure on a double conic 
domain $\XX^{d+1}$ is studied in \cite{X21a}, 
where
$$
\XX^{d+1} = \left\{(\xb, t) \in \RR^{d+1}: \|\xb\| \le |t|, \,\, \xb \in \RR^d, \, t \in [-1,1] \right\}.
$$
It follows that OPs for a weight function on the double cone are divided into two subspaces, depending 
on whether they are even or odd in the $t$-variable, of different characteristics. In particular, for the weight 
function 
$$
   \Wb_{\b,\g}(\xb, t) = |t| \left(t^2 -\|\xb\|^2\right)^{\b - \f12} \left(1-t^2\right)^\g, \quad (\xb,t) \in \XX^{d+1}, 
$$
OPs that are even in the $t$-variable enjoy two characteristic properties of classical OPs. The same phenomenon 
also holds for double hyperbolic domains, and it motivates the recent study in \cite{X24}, where the domain of 
revolution is formulated as from rotating a two-dimensional region $\Omega$, which can be further reduced, when 
$\Omega$ is fully symmetric, to the domain $\Lambda$ so that $\{(x^2,y^2): (x,y) \in \Lambda\}$ is in the first 
quadrant of $\Omega$. In the case of the double cone, $\Lambda$ is the right-angled triangle $\{(s,t): 0 \le s \le t \le 1\}$.
In this setting, OPs on the domain of revolution can be deduced from those of two variables on $\Lambda$ and 
spherical harmonics. In particular, by considering different types of triangles, OPs that are even
in $t$-variable are shown to share the two characteristic properties of classical OPs for a family of domains 
of revolution in \cite{X24}. 

The purpose of the present paper is to reveal a hidden relation between several families of OPs in \cite{X24}, 
generated by those on triangles, to OPs on the unit ball. The relation shows, for example, OPs even in the 
$t$-variable for $\Wb_{\b,\g}$ on $\XX^{d+1}$ are essentially semi-classical, classical if $\b =0$, OPs on the 
unit ball $\BB^{d+1}$ composed with a quadratic transformation. With this relation, it is now easy to establish 
the two characteristic properties for OPs on a family of domains of revolution. This is somewhat unexpected 
since the ball is homogeneous with a smooth boundary, but the conic domains are not. Still, it may not be as 
surprising in retrospect, once the relation is understood. The correspondence has the effect of reducing the 
analysis on the double cone and its generations to the familiar ground of the unit ball. For example, 
approximation and localized polynomial frames on the double cone and hyperboloid were studied in \cite{X23a}, 
as a special case of the general framework developed in \cite{X21} based on highly localized kernels. While 
establishing the localized kernels took substantial efforts in \cite{X23a}, it can now be deduced from the 
existing result on the unit ball effortlessly. 

Restriction to OPs even in the $t$-variable may seem to be incomplete for approximation or Fourier orthogonal
expansions, but it is sufficient for functions that are even in the $t$-variable. Moreover, the fully symmetric
domain has a symmetry in the $t$-variable, so that we can consider only the upper half of the domain, call
it $\VV^{d+1}$. For the double cone $\XX^{d+1}$, for example, the upper half domain $\VV^{d+1}$ is the 
single cone. Given a function $f$ defined on $\VV^{d+1}$, we can extend it to the full domain by defining 
$f(\xb,-t) = f(\xb,t)$ evenly in the $t$-variable. Consequently, OPs even in the $t$-variable consist of a 
complete basis for $L^2$ space on $\VV^{d+1}$, which leads to an alternative Fourier orthogonal expansion
and a complete set for polynomial approximation that works for the spaces equipped with evenly symmetric
weight functions. This orthogonal structure is fundamentally different from the $L^2$ space on $\VV^{d+1}$ 
equipped with the Jacobi weight $(t^2-\|\xb\|^2)^{\b-\f12}(1-t)^\g$, studied in \cite{X20}, which also shares 
the two characteristic properties and has OPs of full dimension. 

The main ingredient for the hidden relation lies in the OPs of two variables on the fully symmetric domains
arising from the triangles, which extends the relation between semi-classical OPs on the disk and the 
classical OPs on the triangle. While the latter relation has been considered, see, for example,
\cite{DX, X24}, the space of OPs that are even in one of the variables has 
hardly been singled out and studied as its own entity in the literature. As a result, we shall treat the case
of two variables first and separately in this work. 

The paper is organized as follows. In the next section, we review the basics of OPs of two
variables, including the orthogonal structure in the fully symmetric setting. The particular family of OPs, 
generated by triangles, and the approximation theory of such polynomials are studied in the third section. 
These are extended in the fourth section to domains of revolution. 

\section{Preliminary: orthogonal polynomials of two variables}
\setcounter{equation}{0}

In this preliminary section, we discuss orthogonal polynomials (OPs) of two variables and the Fourier orthogonal
expansions. The first subsection reviews the basics and classical OPs on the triangle and the disk. The second subsection 
discusses OPs for a fully symmetric weight function. 

\subsection{OPs of two variables}
Let $\Omega$ be a domain in $\RR^2$ with a positive area. Let $\sW$ be a weight function on $\Omega$. 
We consider OPs under the inner product 
\begin{equation} \label{eq:ipd-2d}
  \la f, g\ra_\sW = \sb_\sW \int_\Omega f(\ub) g(\ub) \sW(\ub) \d \ub, \qquad \ub = (u_1,u_2), 
\end{equation}
where $\sb_\sW$ is the normalization constant of $\sW$ such that $\la 1, 1 \ra_\sW = 1$. If the weight function
$\sW$ is regular, that is, $\int_\Omega \|\ub\|^n \sW(\ub)\d \ub < \infty$ for $n \in \NN_0$, then OPs exist for 
all $n$. Let $\CV_n = \CV_n(\sW, \Omega)$ denote the space of OPs of degree $n$ for $n \in \NN_0$. Then
the space $\Pi_n^2$ of polynomials of degree $n$ in two variables satisfies 
$\Pi_n^2 = \bigoplus_{k=0}^n \CV_k(\sW, \Omega)$ and 
$$
\dim \CV_n(\sW,\Omega) = n+1 \quad \hbox{and}\quad \dim \Pi_n^2 = \binom{n+2}{2}, \quad n = 0,1, 2, \ldots. 
$$

Let $\{\sP_{j,n}: 0 \le j \le n\}$ be an orthonormal basis for $\CV_n(\sW,\Omega) $. Then the Fourier orthogonal 
expansion of $f$ is defined by 
$$
  f = \sum_{n=0}^\infty \sum_{j=0}^n \hat f_{j,n} \sP_{j,n}, \quad\hbox{where} \quad \hat f_{j,n} = \la f, \sP_{j,n} \ra_{\sW}.  
$$
The reproducing kernel $ \sP_n(\sW; \cdot,\cdot)$ of the space $\CV_n(\sW, \Omega)$ is defined by 
\begin{equation} \label{eq:reprod_2d}
    \sP_n(\sW; \ub, \vb) = \sum_{j=0}^n \sP_{j,n} (\ub) \sP_{j,n} (\vb),
\end{equation}
which is the kernel of the projection operator $\proj_n(\sW): L^2(\sW, \Omega) \mapsto \CV_n(\sW, \Omega)$, 
$$
    \proj_n (\sW; f, \ub) = \sb_\sW \int_\Omega f(\vb) \sP_n(\sW; \ub,\vb) \sW(\vb) \d \vb. 
$$
For regular $\sW$, the Fourier orthogonal series satisifes 
$$
   L^2(\sW, \Omega) = \bigoplus_{n=0}^\infty  \CV_n(\sW, \Omega): \quad f = \sum_{n=0}^\infty \proj_n(\sW; f). 
$$

We are interested in the cases when an orthogonal basis can be written explicitly in terms of classical OPs of
one variable. Among various classes of such families (cf. \cite{DX, K}), we single out two examples, one on the 
triangle and the other on the disk, which will play essential roles in our study in this paper. 

\subsubsection{Classical OPs on triangle}
On the triangle 
$$
   \triangle = \{(u,v): u \ge 0, v \ge 0, u+v \le 1\},
$$
the classical Jacobi weight function $\sW_{\a}^\triangle$ is defined by, for $\a = (\a_1,\a_2,\a_3)\in \RR^3$,
\begin{equation}\label{eq:Jacobi-w}
     \sW^{\a}_\triangle(u,v) = u^{\a_1} v^{\a_2} (1-u-v)^{\a_3}, \quad \a_i > -1.
\end{equation}
The normalization constant  $\sb^\triangle_{\a}$ of this weight function is given by 
\begin{equation}
   \sb^\triangle_{\a} = \left[ \int_\triangle  \sW^{\a}_\triangle(u,v) \d u \d v \right]^{-1} = \frac{\Gamma(|\a|+3)}{\Gamma(\a_1+1)\Gamma(\a_2+1)\Gamma(\a_3+1)},
\end{equation}
where $|\a|:= \a_1+\a_2+\a_3$. OPs for this weight function can be given explicitly via the Jacobi polynomials
$P_n^{\a,b)}$ of one variable, 
which are orthogonal with respect to the weight function 
$$
   w_{a,b}(t) = (1-t)^a ( 1+ t)^b, \qquad a, b > -1
$$
on the interval $[-1,1]$. In particular an explicit orthogonal basis of $\CV_m(\sW_{\a}^\triangle, \triangle)$ consists of 
polynomials $\{\sT_{j,m}^{\a}: 0 \le j \le m\}$, where \cite[Section 2.4]{DX} 
\begin{equation}\label{eq:triOP}
  \sT_{j,m}^{\a} (u,v) = P_{m-j}^{(2j+ \a_1+\a_3+1,\a_2)}(2 v -1) (1-v)^j P_j^{(\a_3,\a_1)}\left( \frac{2u}{1-v} -1\right).
\end{equation}  
These are classical orthogonal polynomials in two variables and possess many interesting properties. We mention
two of them. The first one is a closed-form formula for the reproducing kernel $\sP_n(\sW^{\a}_\triangle; \cdot,\cdot)$, 
 \begin{align} \label{eq:reprodT}
   \sP_n\left(\sW^{\a}_\triangle; \ub, \vb\right) = \,& c_\a  
       \int_{[-1,1]^3} Z_{2n}^{|\a|+2} \left(\sqrt{u_1}\sqrt{v_1}\, t_1 +\sqrt{u_2}\sqrt{v_2}\, t_2 + \sqrt{1-|\ub|}\sqrt{1-|\vb}\, t_3\right) 
       \notag \\
         & \times (1+t_1)(1-t_1^2)^{\a_1-\f12} (1+t_2) (1-t_2^2)^{\a_2-\f12}(1-t_3^2)^{\a_3-\f12} \d \tb,
\end{align} 
where $Z_n^\l$ is defined in terms of the Gegenbauer polynomial $C_n^\l$, 
\begin{equation}\label{eq:Zn}
  Z_n^\l =  \frac{n + \l} {\l} C_n^\l (t),
\end{equation}
and $c_\a = c_{\a_1} c_{\a_2} c_{\a_3}$ with $c_\a = \int_{-1}^1 (1-t^2)^{\a-\f12} \d t$. The identity \eqref{eq:reprodT} holds 
for $\a_i \ge -\f12$ with the understanding that, if one of more $\a_i = -\f12$, then it holds under the limit
\begin{equation}\label{eq:limit}
   \lim_{a \to -\f12 +} c_\a \int_{-1}^1 f(t) (1-t^2)^{\a-\f12} \d t = \frac{f(1) + f(-1)}2.
\end{equation}
The second one states that these OPs are eigenfunctions of a second-order differential operator defined by 
\begin{align} \label{eq:CD-T}
  \CD^{\a}_\triangle:= \, & u(1-u) \partial_{uu} - 2 u v \partial_{uv} + v(1-v) \partial_{vv} \\
           + \, & (\a_1+1 - (|\a| + 3) u ) \partial_u +  (\a_2 + 1 -(|\a|+ 3) v) \partial_v.  \notag
\end{align}
More precisely, 
\begin{align}\label{eq:eigen-T}
    \CD^{\a}_\triangle Y  = -n (n+|\a| + 2) Y,  \qquad Y \in \CV_m\Big(\sW^{\a}_\triangle, \triangle\Big).
\end{align}

\subsubsection{Semi-classical OPs on the unit disk}
On the unit disk $\BB^2 = \{(u,v): u^2+v^2 \le 1\}$, we consider the weight function 
\begin{equation}\label{eq:WBk}
    \sW^{\k}_\BB (u,v) = |u|^{2\k_1} |v|^{2\k_2} \left(1-u^2-v^2\right)^{\k_3}, \quad \k_1, \k_2 > -\tfrac12, \,\, \k_3 > -1.
\end{equation}
In the case of $\k_1 = \k_2 = 0$, they are the classical OPs, which have been extensively studied. OPs for
$\sW^\k_\BB$ can be given in terms of the generalized Gegenbauer polynomials $C_n^{(\l,\mu)}$ that are 
orthogonal with respect to the weight function
$$
    g_{\l,\mu} (t) = |t|^{2 \mu} (1-t^2)^{\l-\f12}, \qquad \mu, \l > -\tfrac12 
$$
on $[-1,1]$, which reduce to the classical Gegenbauer polynomials $C_n^\l$ when $\mu  = 0$. 
 An explicit orthogonal basis of $\CV_m(\sW^{\k}_\BB, \BB^2)$ consists of 
polynomials $\{\sG_{j,n}^{\k}: 0 \le j \le n\}$, where \cite{DX} 
\begin{equation}\label{eq:OP_disk}
  \sG_{j,n}^\k (u,v) = C_{n-j}^{(j + \k_1+\k_3 +1,\k_2)}(v) (1-v^2)^{\f{j}{2}} 
                C_j^{(\k_3+\f12,\k_1)} \bigg(\frac{u}{\sqrt{1-v^2}} \bigg). 
\end{equation}
These polynomials share analogs of the two properties we cited for OPs on the triangle. The first is 
a closed-form formula for the reproducing kernel $\sP_n(\sW^\k_\BB)$ of $\CV_n(\sW^\k_\BB, \BB^2)$ 
given by \cite[Theorem 8.1.16]{DX}
\begin{align} \label{eq:reprodB}
\sP_n(\sW^\k_\BB; \ub,\vb) = &  c_\k \int_{[-1,1]^3} Z_n^{|\k|+1} \left(u_1 v_1 t_1+ u_2 v_2 t_2 + 
  \sqrt{1-|\ub|^2} \sqrt{1-|\vb|^2}\, t_3 \right) \\
 & \times (1+t_1)(1+t_2) (1-t_1^2)^{\k_1-1} (1-t_2^2)^{\k_2-1}  (1-t_3^2)^{\k_3-\f12}\d t, \notag
\end{align}
which holds for $\k_1, \k_2 \ge 0$ and $\k_3 > -\f12$ and under the limit \eqref{eq:limit} when needed. 

The second is a second-order differential-difference operator, $\CD_\k$, defined by
\begin{align}\label{eq:diffB2}
  \CD^\k_\BB =  (1-u^2)\partial_{uu}  &\, - 2 u v \partial_{uv} + (1-v^2)\partial_{vv} - (2 |\k| + 3) (u \partial_u + v \partial_v) \\
     & + \k_1 \left(\frac 2 u \partial_u - \frac{1-\s_1}{u^2} \right)+\k_2 \left(\frac 2 v \partial_v - \frac{1-\s_2}{v^2} \right), \notag
\end{align}
where $|\k|  = \k_1+\k_2+\k_3$, and $\s_i$ denotes the reflection operator defined by
$$
  \s_1 f(u,v) = f(-u,v) \quad \hbox{and} \quad   \s_2 f(u,v) = f(u,-v), 
$$
that has OPs for $\sW_\k^\BB$ as eigenfunctions; more precisely, 
\begin{equation}\label{eq:eigenB2}
   \CD^\k_\BB Y = - n (n+ 2 |\k| + 2) Y, \qquad Y\in \CV_n(\sW^\k_\BB, \BB^2).
\end{equation}
This is stated in \cite[Theorem 8.1.3]{DX} in terms of the Dunkl Laplacian $\Delta_h$ with $d =2$, which is defined in 
\cite[(8.1.6)]{DX}. In particular, if $\k_1 = \k_2 =0$, then $\CD^\k_\BB$ reduces to the second-order differential operator
for classical OPs on the disk. 

It is worth mentioning that, according to the classification in \cite{KS}, there exists a second-order linear differential 
operator with polynomial coefficients having OPs in two variables as eigenfunctions, up to an affine transform,
only in five cases, three are cross products of Laguerre and Hermite polynomials, and the other two are the classical 
OPs on the triangle and the unit disk given above. The classification does not cover the semi-classical OPs on the disk,
for which the operator contains a difference part. 
 
\subsection{Fully symmetric OPs in two variables} 
We are interested in OPs when the domain and the weight function are fully symmetric. 

\begin{defn}
The domain $\Omega$ is called fully symmetric if $(u, v) \in \Omega$ implies $(\pm u, \pm v) \in \Omega$, and
the weight function $\sW$ defined on such a domain is called fully symmetric if $\sW(\pm u, \pm v) = 
\sW(u,v)$ for $(u,v) \in \Omega$.  
\end{defn}

\subsubsection{Structure of fully symmetric OPs}
A fully symmetric weight $\sW$ is determined by its values on the positive quadrant 
$\Omega_{+,+} = \{(u,v) \in \Omega: u \ge 0, v \ge 0\}$ and it is even in both variables. Hence, it can be written as 
\begin{equation} \label{eq:W=w}
  \sW(u,v) = \sw\left(u^2, v^2\right), \quad (u,v) \in \Omega,
\end{equation}
where the weight function $\sw$ is defined on the domain
\begin{equation} \label{eq:sqrtOmega}
   \sqrt{\Omega} = \left\{(s,t): s = \sqrt{u}, \, t= \sqrt{v},   \,\, (u,v) \in \Omega_{+,+} \right\}.  
\end{equation}
OPs for a fully symmetric weight function can be derived by four families of OPs on $\sqrt{\Omega}$ that 
are orthogonal with respect to $\sw_{\pm \f12, \pm \f12}$ defined by 
\begin{equation} \label{eq:w1212}
  \sw_{\pm \f12}(s,t) = s^{\pm \f12} t^{\pm \f12} \sw(s,t), \qquad (s,t) \in  \sqrt{\Omega}.
\end{equation}
Moreover, OPs in the space $\CV_n(\sW,\Omega)$ inherit symmetry that can be described by  
$$
  \CV_{2m}(\sW,\Omega) = \CV_{2m}^{\sE,\sE}(\sW,\Omega) \cup \CV_{2m}^{\sO,\sO}(\sW,\Omega) \,\, \hbox{and}\,\,
    \CV_{2m+1}(\sW,\Omega) = \CV_{2m+1}^{\sE,\sO} (\sW,\Omega)\cup \CV_{2m+1}^{\sO,\sE}(\sW,\Omega),
$$
where $\CV_{2m}^{\sE,\sE}(\sW,\Omega)$ consists of OPs even in both variables in $\CV_{2m}(\sW,\Omega)$, 
and $\CV_{2m+1}^{\sE,\sO}(\sW,\Omega)$ consists of OPs even in the first variable and odd for the second 
variable in $\CV_{2m+1}(\sW,\Omega)$, for example. The other cases are defined similarly.

Let $\sb(\sw)$ denote the normalization of $\sw$. Then it is easy to see that $\sb(\sw_{-\f12,-\f12}) = \sb_\sW$ for $\sW$
in \eqref{eq:W=w}. 

\begin{thm} \label{thm:FullSymB2}
Let $\Omega$ and $\sW$ be fully symmetric. Let $\{\sP_{j,m}\big(\sw_{\pm \f12, \pm \f12}\big): 0 \le j \le m\}$ 
be an orthonormal basis under the inner product \eqref{eq:ipd-2d} defined via
$\sw_{\pm \f12, -\f12}$ on $\sqrt{\Omega}$, 
Then 
\begin{align*}
    \sP_{j,2m}^{\sE,\sE}(\sW; u,v) & = \sP_{j,m}\left(\sw_{-\f12, -\f12}; u^2,v^2 \right), \quad 0 \le j \le m,\\
    \sP_{j,2m}^{\sO,\sO}(\sW; u,v) & = \frac{\sqrt{\sb\big(\sw_{\f12, \f12}\big)}}{\sqrt{\sb_\sW}} 
       u v \sP_{j,m-1}\left(\sw_{\f12, \f12}; u^2,v^2\right),  \quad 0 \le j \le m-1
 \end{align*}
consist of an orthonormal basis for  $\CV_{2m}^{\sE,\sE} (\sW, \Omega)$ and $\CV_{2m}^{\sO,\sO} (\sW, \Omega)$ respectively,
and 
\begin{align*}
    \sP_{j,2m+1}^{\sE,\sO} (\sW; u,v) & = \frac{\sqrt{\sb\big(\sw_{-\f12, \f12}\big)}}{\sqrt{\sb_\sW}} v \sP_{j,m}\left(\sw_{-\f12, \f12}; u^2,v^2\right), \quad 0 \le j \le m, \\
    \sP_{j,2m+1}^{\sO,\sE} (\sW; u,v) & = \frac{\sqrt{\sb\big(\sw_{\f12, -\f12}\big)}}{\sqrt{\sb_\sW}} u \sP_{j,m}\left(\sw_{-\f12, \f12}; u^2,v^2\right),
    \quad 0 \le j \le m.
 \end{align*}
consist of an orthonormal basis for  $\CV_{2m+1}^{\sE,\sO} (\sW, \Omega)$ and $\CV_{2m+1}^{\sO,\sE}(\sW, \Omega)$
respectively.
\end{thm} 

\begin{proof}
The orthogonality of these polynomials follows from parity and simple change of variables; see \cite[Theorem 4.2]{X24}. 
To see the polynomials are orthonormal, consider $\sP_{2m}^{\sE,\sO}(\sW; u,v)$ as an example. Let 
$\Omega_{+,+} = \{(u,v) \in \Omega: u \ge 0, v \ge 0\}$. By symmetry and changing variables $s = u^2$ and $t=v^2$, 
we obtain
\begin{align*}
   \Big \langle \sP_{2m+1}^{\sE,\sO}(\sW),  \sP_{2m+1}^{\sE,\sO}(\sW) \Big \rangle_{\sW}  
   &=  4 \sb \big(\sw_{-\f12, \f12}\big)\int_{\Omega_{+,+}} v^2 \Big | \sP_{j,m}\Big(\sw_{-\f12, \f12}; u^2,v^2\Big)\Big|^2 \sw\big(u^2,v^2\big) \d u \d v  \\
 & =  \sb\big(\sw_{-\f12, \f12}\big) \int_{\sqrt{\Omega}} s^{-\f12} t^{\f12} \Big | \sP_{j,m}\Big(\sw_{-\f12, \f12}; s,t\Big) \Big|^2 \sw(s,t)
   \d s \d t   \\  
  & =   \Big \langle  \sP_{j,m}\Big(\sw_{-\f12, \f12}\Big), \sP_{j,m}\Big(\sw_{-\f12, \f12})\Big \rangle_{\sw_{-\f12, \f12}} = 1.  
\end{align*}
A similar proof works for the norms of $\sP_{2m+1}^{\sO,\sE}(\sW)$ and $\sP_{2m+1}^{\sO,\sO}(\sW)$.
\end{proof} 

We are particularly interested in the subspace $\CV_n^{\circ, \sE}$ that consists of OPs even in the second variable in 
$\CV_n$, and the subspace $\CV_n^{\sE, \circ}$ that consists of OPs even in the first variable in $\CV_n$. By definition, 
$$
  \CV_{2m}^{\circ, \sE}(\sW, \Omega) = \CV_{2m}^{\sE,\sE} (\sW, \Omega) \quad\hbox{and}\quad
    \CV_{2m+1}^{\circ, \sE}(\sW, \Omega) = \CV_{2m+1}^{\sO,\sE} (\sW, \Omega),
$$
and
$$
  \CV_{2m}^{\sE,\circ}(\sW, \Omega) = \CV_{2m}^{\sE,\sE} (\sW, \Omega) \quad\hbox{and}\quad
    \CV_{2m+1}^{\sE,\circ}(\sW, \Omega) = \CV_{2m+1}^{\sE,\sO} (\sW, \Omega),
$$
In particular, it follows that
$$
  \dim \CV_n^{\circ, \sE} =    \dim \CV_n^{\sE, \circ}  = \left \lfloor \frac{n}2 \right \rfloor + 1.
$$
Because of symmetry, we shall consider mainly the space $\CV_{n}^{\circ, \sE}(\sW)$ in the following. 

Although the space $\CV_n^{\circ, \sE}$ consists of only half of the OPs in $\CV_n$, it is nevertheless sufficient 
for the Fourier orthogonal expansions and approximation on the domain 
\begin{equation} \label{eq:Lambda-gen}
   \Lambda := \Omega^{\circ, \sE} := \{(u,v) \in \Omega: v \ge 0\},
\end{equation}
which is the upper half of the fully symmetric domain $\Omega$, equipped with the weight function $\sW$.
Let $\sP_n^{\circ, \sE}(\sW; \cdot,\cdot)$ be the reproducing kernels of $\CV_n^{\circ, \sE}$. By definition,
$$
   \sP_n^{\circ, \sE}(\sW; \ub,\vb) = \sum_{j=0}^{ \left \lfloor \frac{n}2 \right \rfloor + 1} \sQ_{j,n}(\ub) \sQ_{j,n}(\vb),
$$
where $\sQ_{j,n}$ consists of an orthonormal basis of  $\CV_n^{\circ, \sE}$. Moreover, if $f$ is even in its 
second variable, then its orthogonal projection operator 
$$
\proj_n^{\circ, \sE} : L^2(\sW, \Lambda) \mapsto \CV_n^{\circ, \sE}
$$
satisfies the relation
$$
  \proj_n^{\circ, \sE}(\sW; f, \ub) = \sb_\sW \int_\Lambda f(\vb) \sP_n^{\circ, \sE}(\sW; \ub,\vb) \sW(\vb) \d \vb.
$$

\begin{thm} \label{thm:Fourier}
Let $\Omega$ and $\sW$ be fully symmetric. If $f \in L^2(\sW, \Lambda)$, then 
\begin{equation}\label{eq:Fourier}
       f = \sum_{n=0}^\infty  \proj_n^{\circ, \sE}(\sW; f). 
\end{equation}
\end{thm}

\begin{proof}
The kernel $\sP_n^{\circ, \sE}(\sW)$ can be written in terms of the reproducing kernel $\sP_n(\sW)$ of $\CV_n(\sW)$ as 
\begin{equation}\label{eq:PnBE1}
   \sP_n^{\circ, \sE}(\sW; \ub,\vb) = \frac12 \left[ \sP_n \big (\sW; (u_1,u_2), (v_1,v_2)\big) 
                + \sP_n \big(\sW; (u_1,u_2), (v_1, - v_2)\big) \right],
\end{equation}
which follows since the integral kernel in the right-hand side reproduces all polynomials in $\CV_n(\sW,\Omega)$ that 
are even in the second variable. For $f$ defined on $\Lambda$, we can extend it to the fully symmetric domain 
$\Omega$ by defining $f(u,-v) = f(u,v)$ for $(u,v) \in \Lambda$. Evidently, $f\in L^2(\Lambda, \sW)$ is equivalent to 
$f\in L^2(\sW,\Omega)$. For the extended $f$ on $\Omega$, the identity \eqref{eq:PnBE1} leads to the relation
\begin{align*}
   \proj_n^{\circ, \sE}(\sW; f, \ub) \, & = \f{\sb_\sW}{2} \int_\Lambda f(\vb)  \frac12 \left[ \sP_n \big (\sW; \ub, \vb\big) 
                + \sP_n \big(\sW; \ub, (v_1, - v_2)\big) \right] \sW(\vb) \d \vb  \\
           &  = \sb_\sW \int_\Omega f(\vb) \sP_n \big (\sW; \ub, \vb\big) \sW(\vb) \d \vb = \proj_n (\sW; f,\ub). 
\end{align*}
Hence, the stated identity is simply the restriction of the Fourier orthogonal expansion to $f$ that is even
in the second variable. 
\end{proof}

\begin{rem}
The space $L^2(\Lambda, \sW)$ contains a basis of OPs, so that it has the Fourier orthogonal expansion according to 
the decomposition 
$$
L^2(\sW,\Lambda) = \sum_{n=0}^\infty \CV_n(\sW,\Lambda), \qquad \dim \CV_n(\Lambda, \Omega) = n+1. 
$$
The orthogonal expansion stated in the theorem, however, is different. It is the restriction of the Fourier orthogonal 
expansion of $L^2(\sW,\Omega)$ on $\Lambda = \Omega^{\circ, \sE}$ for functions that are even in the second 
variable. The second one may look somewhat artificial, but it can be useful when an explicit basis for
$\CV_n^{\circ,\sE}(\sW,\Lambda)$ can be constructed, whereas such a basis is not available for $\CV_n(\sW,\Lambda)$, 
as we shall show in the next section. 
\end{rem}

 
\subsubsection{Fully symmetric orthogonal polynomials on the unit disk}
The unit disk $\BB^2$ and the semi-classical weight function $\sW^{\k}_\BB$ at \eqref{eq:WBk} are fully symmetric,
and the classical OPs on the triangle generate their OPS. Indeed, if $\Omega = \BB^2$, then $\sqrt{\Omega} = \triangle$ 
and, with $\sw(u,v) = \sW^\triangle_{\k_1,\k_2,\\k_3}(u,v)$, $\sw_{\pm \f12,\pm\f12}$ in \eqref{eq:w1212} becomes
$$
  \sw_{\pm \f12,\pm \f12} (u,v) = u^{\k_1\pm 
  \f 12} v^{\k_2 \pm \f12} (1-u-v)^{\k_3}.    
$$
In particular, by Theorem \ref{thm:FullSymB2}, an orthogonal basis for $\CV_n^{\circ, \sE}$ is given by
\begin{align}\label{eq:OP_B-T1}
\begin{split}
    \sP_{j,2m}^{\sE,\sE}(u,v)\, & = \sT_{j,m}^{\k_1-\f12, \k_2-\f12, \k_3} \big(u^2,v^2\big), \quad 0 \le j \le m, \\   
  \sP_{j,2m+1}^{\sO, \sE}(u,v)\, & = u \sT_{j,m}^{\k_1+\f12, \k_2-\f12, \k_3} \big(u^2,v^2\big), \quad 0 \le j \le m, 
\end{split}
\end{align}
and an orthogonal basis for $\CV_n^{\circ,\sO}$ is given by
\begin{align}\label{eq:OP_B-T2}
\begin{split}
    \sP_{j,2m}^{\sO,\sO}(u,v)\, & = uv \sT_{j,m}^{\k_1+\f12, \k_2+\f12, \k_3} \big(u^2,v^2\big), \quad 0 \le j \le m, \\    
    \sP_{j,2m+1}^{\sE,\sO}(u,v)\, & = v \sT_{j,m}^{\k_1-\f12, \k_2+\f12, \k_3} \big(u^2,v^2\big), \quad 0 \le j \le m.
\end{split}
\end{align}
Up to a multiple constant, this basis agrees with the basis $\sG_{j,n}^\k$ in \eqref{eq:OP_disk}. For better
reference, we state this formerly for the basis of $\CV_n^{\circ,\sE}(\sW_\k, \BB^2)$. 

\begin{prop}  \label{prop:OP_G}
Let $\sG_{j,n}^\k$ be defined in \eqref{eq:OP_disk}. Then, for $0 \le j \le m$, 
$$
  \sG_{2j,n}^\k (u,v) = \mathrm{cons.} \begin{cases}  \sT_{j,m}^{\k_1-\f12, \k_2-\f12, \k_3} \big(u^2,v^2\big), & n = 2m \\
   u \sT_{j,m}^{\k_1+\f12, \k_2-\f12, \k_3} \big(u^2,v^2\big), & n = 2m +1 .\end{cases} 
$$
Moreover, $\{ \sG_{2j,n}^\k (u,v): 0 \le j \le n/2\}$ is an orthogonal basis fo $\CV_n^{\circ,\sE}(\sW_\k, \BB^2)$.
\end{prop}

\begin{proof}
The generalized Gegenbauer polynomials $C_n^{(\l,\mu)}$ are given explicitly in terms of the Jacobi polynomials 
\cite[Section 1.5.2]{DX}  
\begin{align} \label{eq:gGegen}
\begin{split}
C_{2m}^{(\l,\mu)}(t) &\, = \frac{(\l+\mu)_m}{(\mu+\f12)_m} P_m^{(\l-\f12,\mu-\f12)}(2 t^2 -1),\\
C_{2m+1}^{(\l,\mu)}(t) &\, = \frac{(\l+\mu)_{m+1}}{(\mu+\f12)_{m+1}} t P_m^{(\l-\f12,\mu+\f12)}(2 t^2 -1),
\end{split}
\end{align}
which allows us to deduce the stated identity from \eqref{eq:triOP} and \eqref{eq:OP_disk}. 
\end{proof}

The identification of the two bases can also be used to show that the closed-form formula \eqref{eq:reprodT} of the
reproducing kernels on the triangle can be deduced from \eqref{eq:reprodB} for the reproducing kernels on the disk. 
Moreover, the spectral relation \eqref{eq:eigen-T} can be deduced from the relation \eqref{eq:diffB2}. 

We end this section by specializing the two properties for fully symmetric OPs to those in $\CV_n^{\sE,\circ}$ 
and state them as a proposition for later reference. 

\begin{prop}
Let $\k = (0, \k_2 ,\k_3)$, so that $\sW_\k (u,v) = |v|^{2\k_2} \left(1-u^2-v^2\right)^{\k_3}$. Then the OPs in the space 
$\CV_n^{\circ,\sE}(\sW_\k, \BB^2)$ satisfy \eqref{eq:eigenB2} for the differential operator
\begin{align}\label{eq:diffB2E}
  \CD_\k^\BB =  (1-u^2)\partial_{uu}  &\, - 2 u v \partial_{uv} + (1-v^2)\partial_{vv} - (2 |\k| + 3) (u \partial_u + v \partial_v) 
  +  \frac {2 \k_2 } v \partial_v.
\end{align}
Moreover, the reproducing kernel $\sP_n^{\sE,\circ}(\sW_\k; \cdot,\cdot)$ of $\CV_n^{\sE,\circ}(\sW_\k, \BB^2)$ satisfies 
\begin{align} \label{eq:reprodBE}
\sP_n(\sW_\k^\BB; \ub,\vb) = &\, c_\k \int_{[-1,1]^2} Z_n^{|\k|+1} \left(u_1 v_1+ u_2 v_2 t_2  + 
  \sqrt{1-|\ub|^2} \sqrt{1-|\vb|^2}\, t_3 \right) \\
 &\qquad\qquad\qquad\qquad  \times (1-t_2^2)^{\k_2-1} (1-t_3^2)^{\k_3-\f12}\d t. \notag
\end{align}
\end{prop} 

Indeed, since $Y \in \CV_n^{\circ.\sE}(\sW_\k, \BB^2)$ is even in its second variable, $(1- \s_2) Y = 0$, so that the
difference part of $\CD_\k^\BB$ for $\k_2$ in \eqref{eq:diffB2} becomes zero, and the operator becomes \eqref{eq:diffB2E}. 
 Moreover, the identity \eqref{eq:reprodB} for the reproducing kernel holds under limit $\k_1 \to 0$, and 
 \eqref{eq:reprodBE} follows from \eqref{eq:PnBE1} by changing variable $t_2 \to - t_2$ in the integral representation 
 for the second term on the right-hand of \eqref{eq:PnBE1}. A similar statement can be given for 
 $\CV_n^{\sE,\circ}(\sW_\k, \BB^2)$ if $\k = (\k_1,0,\k_3)$. 
 
It is worth to emphasis that the operator $\CD_\k^\BB$ in \eqref{eq:diffB2E} is a differnetial operator, for which 
$\CV_n^{\circ,\sE}(\sW_\k, \BB^2)$ is its space of eigenfunctions, but $\CV_n(\sW_\k, \BB^2)$ is not. 
  
\section{Approximation and OPs on fully symmetric domains}
\setcounter{equation}{0}
 
Throughout the rest of the paper, we let $\fa, \fb, \fc$ be real numbers satisfying 
$$
0 \le \fa < \fb \quad\hbox{and}\quad \fc \ge 0.
$$ 
 
\subsection{OPs in two variables}
We denote by $\triangle_{\fa,\fb,\fc}$ the triangle that has vertices 
$$
(0,\fa), \quad (0,\fb), \quad (1,\fc) \quad \hbox{with} \quad  0 \le \fa < \fb, \quad \fc \ge 0,
$$ 
which is in the first quadrant of the plane and given algebraically as
$$
\triangle_{\fa,\fb,\fc} = \{(u,v): \fa + (\fc-\fa) u \le v \le  \fb + (\fc-\fb) u, \quad 0 \le u \le 1\}.
$$
For $\fa > 0$, the triangles are depicted in Figure 1. The triangle is affine transformed to the triangle $\triangle$ by 
\begin{equation} \label{eq:tri-tri}
    (u,v) \in \triangle_{\fa,\fb,\fc} \mapsto (s,t) \in \triangle: \quad s = u, \quad t = \frac{-\fa+(\fa-\fc)u+v}{\fb-\fa}.  
\end{equation}

\begin{figure}[htb]
\begin{minipage}{1\textwidth}
\centering
\includegraphics[width=4cm]{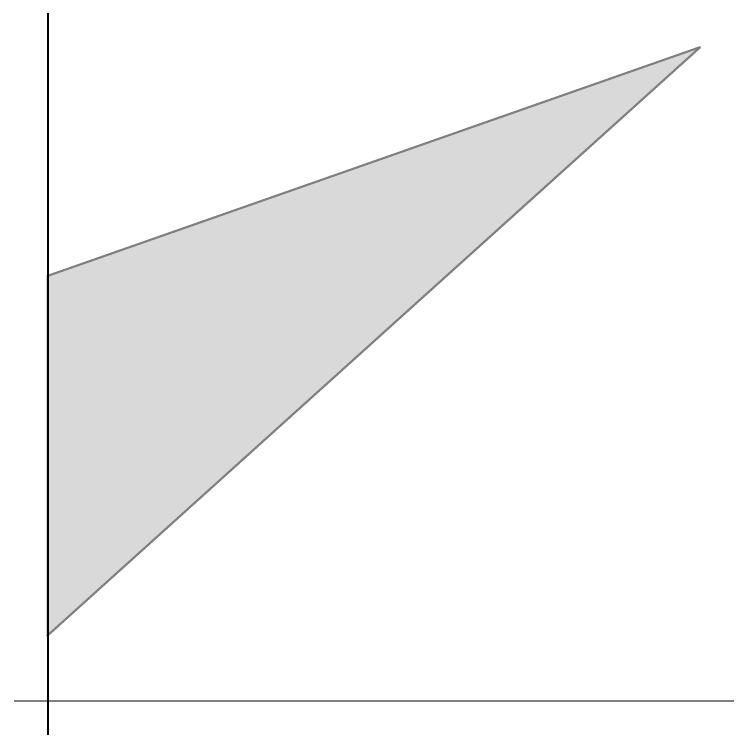}  \quad 
 \includegraphics[width=4cm]{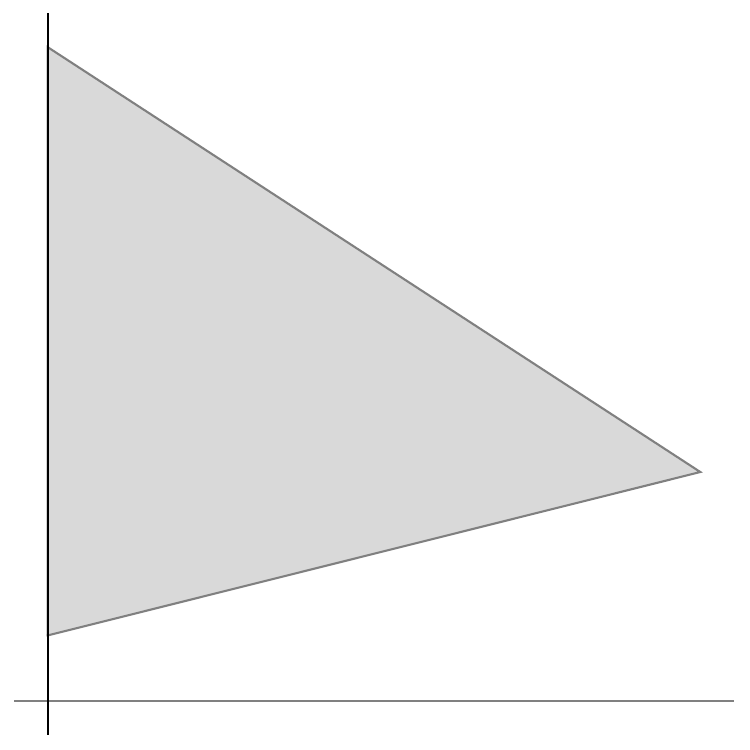} \quad
 \includegraphics[width=4cm]{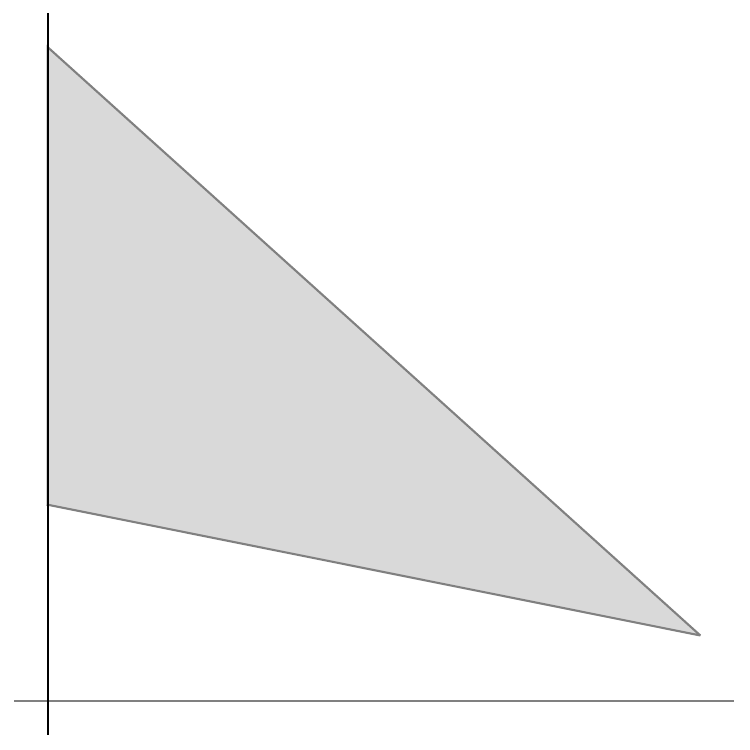}
\end{minipage}
\caption{Left $\fa< \fb< \fc$ \qquad Center $\fa < \fc < \fb$ \qquad Right $\fc < \fa < \fb$ \qquad}
\label{Fig1}
\end{figure}
 
Let $ \sW_\k^\triangle$ be the classical Jacobi weight on the triangle $\triangle$. Denote 
$$
     z(u,v) = \frac{-\fa+(\fa-\fc)u+v}{\fb-\fa}.
$$
On the triangle $\triangle_{\fa,\fb,\fc}$, define the weight function 
\begin{align}\label{eq:w_abc}
  \sw_{\fa,\fb,\fc}^\k (u,v)  &\, = \sW_\k^\triangle\left(u,z(u,v) \right) \\
       &\, =u^{\k_1} \left(\frac{-\fa+(\fa-\fc)u+v}{\fb-\fa}\right)^{\k_2} 
       \left(\frac{\fb-(\fb-\fc)u-v}{\fb-\fa}\right)^{\k_3}.\notag
\end{align}
Since the affine transform preserves orthogonality, it follows readily that an orthogonal basis for 
$\CV_n( \triangle_{\fa,\fb,\fc}, \sw_{\fa,\fb,\fc}^\k)$ is given by 
\begin{equation} \label{eq:OPabc}
  \sP_{j,n}(u,v) = \sT_{j,n}^\k \left(u,z(u,v) \right), \quad 0 \le j \le n,
\end{equation}
which are affine transformations of the standard Jacobi basis on the triangle. 

The fully symmetric domain and weight functions that we are interested in arise from $\triangle_{\fa,\fb,\fc}$ and 
$\sw_{\fa,\fb,\fc}^\k$, and are defined formly below. 

\begin{defn}\label{defn:Oabc}
For $0 \le \fa < \fb$, $\fc \ge 0$, define the fully symmetric domain $\Omega_{\fa,\fb,\fc}$ by 
$$
   \Omega_{\fa,\fb,\fc} = \left \{(u,v): \fa + (\fc-\fa) u^2 \le v^2 \le  \fb + (\fc-\fb) u^2, \quad  |u| \le 1\right \},
$$
and define the weight function $\sW_{\fa,\fb,\fc}^\k$, $\k_1,\k_2 > -\f12$ and $\k_3 > -1$, on $\Omega_{\fa,\fb,\fc}$ by 
\begin{align} \label{eq:sWabcB2}
  \sW_{\fa,\fb,\fc}^\k(u,v) \,& = |v| \sw_{\fa, \fb,\fc}(u^2,v^2) \\
     \, & = |v| |u|^{2 \k_1}  
      \left(\frac{-\fa+(\fa-\fc)u^2+v^2}{\fb-\fa}\right)^{\k_2-\f12} \left(\frac{\fb-(\fb-\fc)u^2-v^2}{\fb-\fa}\right)^{\k_3}. \notag
\end{align}
\end{defn}

The domain $\Omega_{\fa,\fb,\fc}$ is fully symmetric. What we mainly concentrate on is $\Omega_{\fa,\fb,\fc}^{\circ, \sE}$,
which is the portion in the upper half plane as defined in \eqref{eq:Lambda-gen}, and we denote it by
\begin{equation*}
   \Lambda_{\fa,\fb,\fc} := \Omega_{\fa,\fb,\fc}^{\circ, \sE} = \left\{(u,v) \in \Omega_{\fa,\fb,\fc}: v \ge 0 \right \}.
\end{equation*}
For $\fa > 0$, the domains corresponding to those in Figure 1 are depicted in Figure 2. 
\begin{figure}[htb]
\begin{minipage}{1\textwidth}
\centering
\includegraphics[width=4cm]{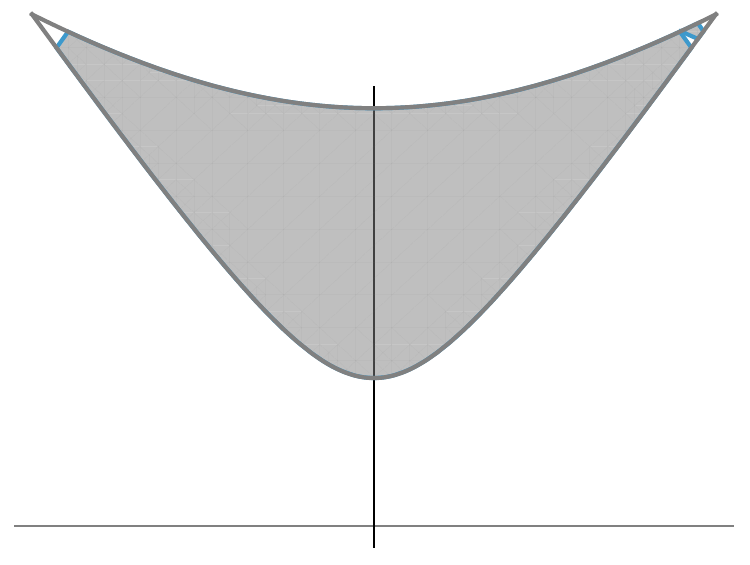}  \quad 
 \includegraphics[width=4cm]{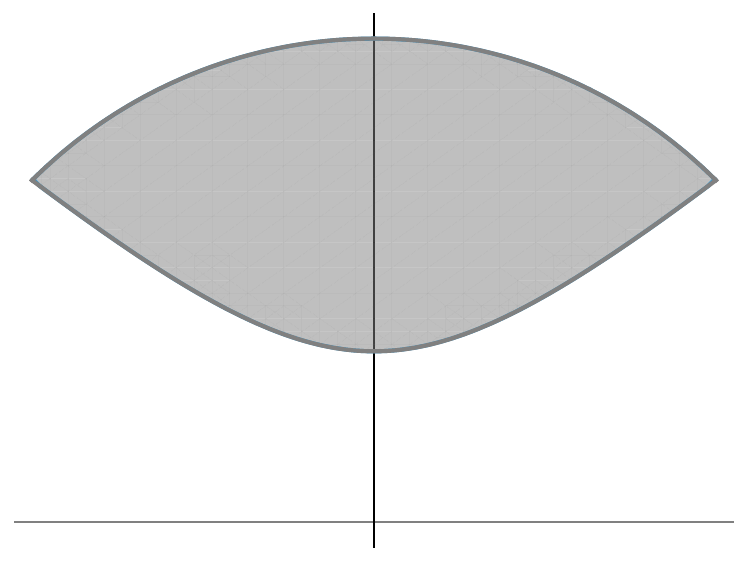} \quad
 \includegraphics[width=4cm]{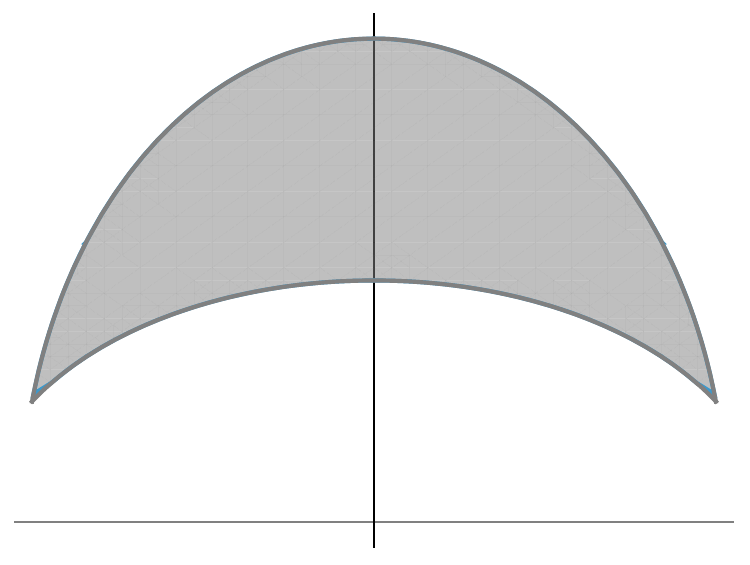}
\end{minipage}
\caption{Left $\fa< \fb< \fc$ \qquad Center $\fa < \fc < \fb$ \qquad Right $\fc < \fa < \fb$ \qquad}
\label{Fig2}
\end{figure}
If $\fc = \fa$, then the lower curve in the last two figures becomes a horizontal line segment, whereas if $\fc = \fb$, 
then the upper curve becomes a horizontal line segment. 

For $\fa =0$, the domain $\Lambda_{0,\fb,\fc}$ corresponding to the third triangle in Figure 1 is the upper half 
of the ellipsoid when $\fc =0$, which becomes the upper disk if $\fb=1$ as well. The first two cases coincide 
when $\fc = \fb$, for which $\Lambda_{0,\fb,\fc} = \{(u,v): |u| \le v \le 1\}$ is a triangle, whereas the two cases 
have different characteristics, when $\fc < \fb$ or $\fc > \fb$, and they are depicted in Figure 3. 
\begin{figure}[htb]
\begin{minipage}{1\textwidth}
\centering
\includegraphics[width=6cm]{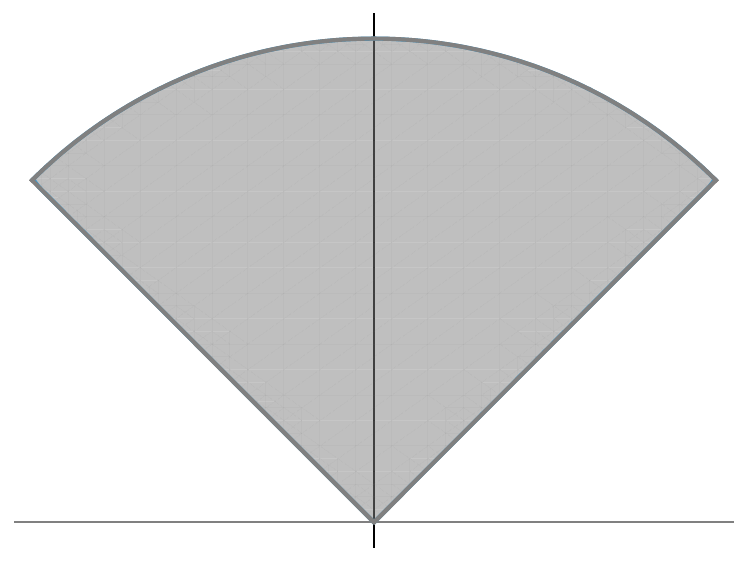}  \qquad 
 \includegraphics[width=6cm]{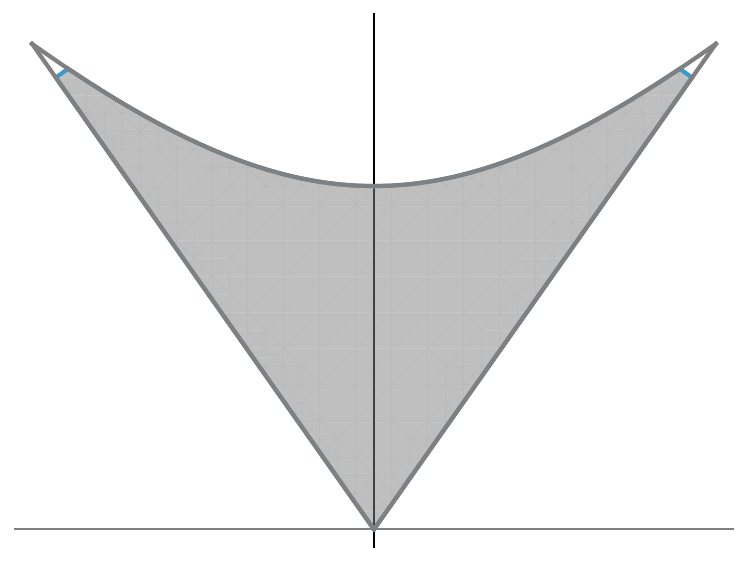}  
\end{minipage}
\caption{Left: $\fb=1,\, \fc = \f12$ \qquad\qquad\quad  Right: $\fb=1,\, \fc = 2$ \qquad}
\label{Fig3}
\end{figure}

The weight function $\sW_{\fa,\fb,\fc}^\k$ is fully symmetric. In the case of $\fa =\fc = 0$ and $\fb=1$, the 
domain $\Omega_{0,1, 0}$ becomes the unit disk and the weight function satisfies $\sW_{0,1,0}^\k(u,v) = \sW_\BB^\k(u,v)$. 
This explains our choice of the parameter $\k_2 -\f12$ instead of $\k_2$ in the definition. 
In terms of the notation of the previous section, we have $\sqrt{\Omega_{\fa,\fb,\fc}} = \triangle_{\fa,\fb,\fc}$. The 
corresponding weight function $\sw$ defined by \eqref{eq:W=w} becomes
\begin{equation} \label{eq:sWabc}
    \sW_{\fa,\fb,\fc}^\k(u,v) = \sw(u^2,v^2) \quad\hbox{with} \quad \sw(u,v)  = v^{\f12} \sw_{\fa,\fb,\fc}^{\k_1,\k_2-\f12, \k_3} (u,v),  
\end{equation}
By Theorem \ref{thm:FullSymB2}, OPs for $\CV_n(\sW_{\fa,\fb,\fc}^\k, \Omega_{\fa,\fb,\fc}^\k)$ can be given 
in terms of OPs with respect to $\sw_{\pm \f12, \pm\f12}$ on the triangle, but not all in terms of classical OPs 
on the triangle. Indeed, 
\begin{align*}
  \sw_{\pm \f12, - \f12}(u,v) \,& = \sW_\triangle^{\k_1 \pm \f12, \k_2-\f12, \k_3}\left(u, z(u,v) \right)
\end{align*}
are classical Jacobi weights on $\triangle$ under \eqref{eq:tri-tri} but $\sw_{\pm \f12, \f12}$ are not unless $\fa = \fc = 0$
because of the factor $v^\f12$. In the latter case, the basis for $v^{\f12} \sw_{\fa,\fb,\fc}^{\k_1,\k_2-\f12, \k_3}$ 
on the triangle can still be written in terms of the Jacobi polynomials (see, e.g. \cite{OTV}), but they are no 
longer the eigenfunctions of a second-order differential operator, nor do they satisfy an addition formula.  

The OPs for $\sW_{\fa,\fb,\fc}^\k$ generated by OPs for the classical Jacobi weight consist of an orthogonal basis 
for the space $\CV_n^{\circ, \sE}(\sW_{\fa,\fb,\fc}^\k, \Omega_{\fa,\fb,\fc}^\k)$. They are given by 
$\sP_{j,2m}^{\sE,\sE}$ for $n = 2m$ and $\sP_{j, 2m+1}^{\sO,\sE}$ for $n = 2m+1$ under the change of variables 
$(u,v) \mapsto (s,t)$ in \eqref{eq:tri-tri}. Hence, by \eqref{eq:OP_B-T1} and \eqref{eq:OP_B-T2} and 
Proposition \ref{prop:OP_G}, an explicit orthogonal basis for $\CV_n^{\circ, \sE}(\sW_{\fa,\fb,\fc}^\k, \Omega_{\fa,\fb,\fc}^\k)$
can be given in terms of semi-classical OPs on the unit disk. 

The connection between our fully symmetric domain and the unit disk is the central piece of our observation. 
To further clarify the connection, we provide a bijection between the upper half of $ \Omega_{\fa,\fb,\fc}$, denoted by
$$
    \Lambda_{\fa,\fb,\fc} := \left \{(u,v) \in \Omega_{\fa,\fb,\fc}: v \ge 0\right\},
$$
and the upper half of the unit disk $\BB_+^2: = \{(s,t) \in \BB^2: t \ge 0\}$. Let 
\begin{equation} \label{eq:tuv}
  \ft(u,v): = z(u^2,v^2) = \sqrt{\frac{-\fa+(\fa-\fc)u^2+v^2}{\fb-\fa}}. 
\end{equation}

\begin{lem} \label{lem:st-uv}
A bijection bewteen $\Lambda_{\fa,\fb,\fc}$ and $\BB_+^2$ is given by 
\begin{equation}\label{st-uv}
\psi: (u,v) \in  \Lambda_{\fa,\fb,\fc,} \mapsto (s,t) \in \BB_+^2, \qquad \psi (u,v) = (u, \ft (u,v)) = (s,t),
\end{equation}
which leads to the integral identity 
$$
   \int_{\Lambda_{\fa,\fb,\fc}} f(u,v) \sW_{\fa,\fb,\fc}^\k (u,v) \d u \d v = 
       \int_{\BB_+^2} (f\circ\psi^{-1})(s,t) \sW_\BB^\k(s,t) \d s \d t. 
$$
\end{lem}

\begin{proof}
It is straightforward to verify that the mapping is a bijection. Computing the Jacobian of the change of variables 
$(u,v) \mapsto (s,t)$, it is easy to see that
$$
v\, \d u \d v = (\fb - \fa)t\, \d s \d t, 
$$
which leads to the identity 
$$
      \sW_{\fa,\fb,\fc}^{\k} (u,v) \d u \d v = s^{2\k_1} t^{2\k_2} (1-s^2-t^2)^{\k_2} \d s \d t= \sW_\BB^\k (s,t) \d s \d t,
$$
so that the integral identity follows from a change of variables. 
\end{proof}

\begin{thm}\label{prop:sQjn}
Let $\sG_{j,n}^\k$ be the semi-classical orthogonal polynomials on the disk given in \eqref{eq:OP_disk}. 
Let $\sQ_{j,n}^\k$ be polynomials defined by 
\begin{align} \label{sQ-sG}
   \sQ_{j,n}^\k(u,v)\, &  = C_{n-2j}^{(2j + \k_2+\k_3 +1,\k_1)}(u) (1-u^2)^{j} 
   P_j^{(\k_3, \k_2-\f12)}\bigg(2 \frac{-\fa+(\fa-\fc)u^2+v^2}{(\fb-\fa)(1-u^2)} -1\bigg) \notag  \\
   \,& = (\sG_{2j,n}^\k \circ \psi) (u,v).
\end{align}
Then $\left\{ \sQ_{j, n}^\k, \,\, 0  \le j \le \left\lfloor  \f{n}{2} \right\rfloor \right\}$ is an 
orthogonal basis for $\CV_n^{\circ, \sE}(\sW_{\fa,\fb,\fc}^\k, \Omega_{\fa,\fb,\fc})$.
\end{thm}

By symmetry, these OPs are also a basis for $\CV_n^{\circ, \sE}(\sW_{\fa,\fb,\fc}^\k, \Lambda_{\fa,\fb,\fc})$. 
The orthogonality follows immediately from the integral identity in the Lemma \ref{lem:st-uv}, and it also follows
from the orthogonality for the fully symmetric setting via the classical OPs on the triangle, as we discussed 
before the lemma. 

Using this explicit basis, we can deduce that OPs in the space $\CV_n^{\circ, \sE}(\sW_{\fa,\fb,\fc}^\k, \Omega_{\fa,\fb,\fc})$
are eigenfunctions of a second-order differential operator when $\k_1 =0$. 

\begin{thm}\label{thm:PDE_d=2}
For $\k_1 =0$, the polynomials in $\CV_n^{\circ, \sE}(\sW_{\fa,\fb,\fc}^\k, \Omega_{\fa,\fb,\fc})$ 
are eigenfunctions of a differential operator, 
\begin{equation} \label{eq:spectral}
    \fD_{\fa,\fb,\fc}^\k Y = -n(n+ 2\k_2+2\k_3+ 2) Y, \qquad Y \in \CV_n^{\circ, \sE}(\sW_{\fa,\fb,\fc}^\k, \Omega_{\fa,\fb,\fc}),
\end{equation} 
and the operator is given explicitly by, for $(u,v) \in \Omega_{\fa,\fb,\fc}$, 
\begin{align*}
 \fD_{\fa,\fb,\fc}^\k \, & = (1-u^2) \partial_{uu} - 2(-\fc+v^2) \frac{u}{v} \partial_{u v} + \left[ \fa + \fb - v^2
       + \frac{- \fa \fb + (\fa - \fc) (\fb - \fc) u^2}{v^2}\right]\partial_{vv}\\
  & + \left[ -2\fa + \fc -\frac{- \fa \fb + (\fa - \fc) (\fb - \fc) u^2}{v^2} + \right]\frac{1}{v}\partial_{v} 
    + (2|\k|+3) \left[ u \partial_u + (v^2-\fa)\frac{1}{v}\partial_{v}\right]\\
  &  + 2 \k_2 (\fb-\fa) \frac{1}{v} \partial_v. 
\end{align*}
Moreover, the operator is related to the second-order differential operator $\CD_{\BB}^\k$ by 
\begin{equation} \label{eq:fD=D}
        \fD_{\fa,\fb,\fc}^\k (f\circ \psi) = \left(\CD_\BB^\k f \right) \circ \psi.   
\end{equation}
\end{thm} 
 
\begin{proof}
Our goal is to find an operator $\fD_{\fa,\fb,\fc}^\k$ that satisfies \eqref{eq:fD=D}. 
Indeed, by \eqref{sQ-sG}, such an operator would imply 
\begin{align*}
  \fD_{\fa,\fb,\fc}^\k \sQ_{j,n}^\k \, & =  \fD_{\fa,\fb,\fc}^\k (\sG_{2j,n}^\k \circ \psi) 
       =  \left(\CD_\BB^\k \sG_{2j,n}^\k \right) \circ \psi \\ 
      &  = -n(n+2|\k|+2) \sG_{2j,n}^\k \circ \psi =  -n(n+2|\k|+2) \sQ_{j,n}^\k,  
\end{align*}
where the second identity uses \eqref{eq:fD=D} and the third one follows from \eqref{eq:eigenB2},
which would have verified \eqref{eq:spectral}. 

Let $F(u,v) = (f\circ \psi)(u,v) = f(u, \ft(u,v))$. Denote by $\partial_1f$ and $\partial_2f$ the derivative of 
$f$ for the first and the second variable. Taking the derivative of $\partial_u F$ and $\partial_v F$
by chain rule and solve for $\partial_1 f$ and $\partial_2 f$ accordingly, it is straightforward to verify that
\begin{align*}
  \partial_2 f(u,\ft(u,v)) \, &= (\fb - \fa) \ft(u,v) \frac{1}{v} \partial_v F(u,v), \\
  \partial_1 f(u, \ft(u,v)) \, &= \partial_u F(u,v) - (\fa - \fc) \frac{u}{v} \partial_v F(u,v).
\end{align*}
Moreover, taking the derivatives one more time and solving for the second-order derivatives of $f$, we
obtain 
\begin{align*}
  \partial_2^2 f(u, t(u,v)) = \, &  (\fb-\fa)^2 \f{\ft(u,v)^2}{v^2} \partial_u^2 F(u,v) - (\fb - \fa)\big(-\fa+ (\fa - \fc)v^2\big)
   \frac{u}{v} \partial_v F(u,v), \\
  \partial_1 \partial_2 f(u, \ft(u,v))  = \, & (\fb - \fa) \frac{\ft(u,v)}{v} \partial_u \partial_v F(u,v)
         + (\fa - \fc)^2  \frac{u^2}{v^2} \partial^2_v F(u,v)\\
             \, &  - (\fa - \fc)(\fb-\fa)  \frac{u \ft(u,v)}{v^3} \partial_v F(u,v), \\
    \partial_1^2 f(u, \ft(u,v)) = \, & \partial_u^2 F(u,v) -2 (\fa - \fc) \frac{u}{v} \partial_u \partial_v F(u,v) 
        + (\fa-\fc)^2  \frac{u^2}{v^2} \partial^2_v F(u,v) \\
     & +  (\fa - \fc) \big((\fa - \fc) u^2 + v^2\big) \frac{1}{v^3} \partial_v F(u,v).       
\end{align*}
Putting these together and substituting them into $\CD_\BB^\k$ that has the coefficients, in front of the differentials,
in the variable $(s,t) = (u,t(u,v)))$, we obtain the stated expression for $\fD_{\fa,\fb,\fc}^\k$ after simplification by
\eqref{eq:fD=D}.  
\end{proof}

In the case $\fa = 0, \, \fb=1, \,\fc=0$, the operator $\fD_{0,1,0}^\k$ is exactly $\CD_\BB^\k$ with $\k_1=0$,
which is the second-order differential operator with polynomial coefficients. For all other cases, the 
coefficients of the operator $\fD_{\fa,\fb,\fc}^\k$ are not polynomials but rational functions with the power
of $v$ in the denominators. 

Along the same line, we can also state an addition formula for $\sP_n^{\circ,\sE}(\sW_{\fa,\fb,\fc}^\k)$, with $\k_1=0$,
by making a change of variable $u_2 \to \ft(u_1,u_2)$ and $v_2 \to \ft(v_1,v_2)$ in \eqref{eq:reprodBE} since
$$
\sP_n^{\circ,\sE}\left(\sW_{\fa,\fb,\fc}^\k,\; (\ub,\vb)\right) = \sP_n\left(\sW_\BB^\k; \big(u_1, \ft(u_1,u_2)\big),
   \big(v_1, \ft(v_1,v_2)\big) \right). 
$$
 
\subsection{Approximation on the curved domain}
Based on the discussion in the previous section and Theorem \ref{thm:Fourier}, much of the analysis 
of the Fourier orthogonal expansion and approximation by polynomials on the domain $\Lambda_{\fa,\fb,\fc}$
can be deduced from the corresponding results on the unit disk. We state a couple of results on approximation
by polynomials as examples. 

Since $\k_1 = 0$ for OPs on $\Lambda_{\fa,\fb,\fc}$, we replace $\k$ by $(\b,\g)$ for $\b > -\f12$ and $\g > -1$ in 
this subsection. Let $\Pi_n^{\circ,\sE}$ be the space of polynomials of degree at most $n$, in two variables, 
that are even in the second variable. Then 
$$
   \Pi_n^{\circ,\sE} = \bigoplus_{k=0}^n \CV_k(\sW_{\fa, \fb, \fc}^{\b,\g},\Lambda_{\fa, \fb, \fc}). 
$$
The Fourier orthogonal series \eqref{eq:Fourier} holds if the $n$-th partial sum 
$$
  S_n^{\circ,\sE} (\sW_{\fa, \fb, \fc}^{\b,\g}; f) = \sum_{k=0}^n \proj_k^{\circ,\sE} (\sW_{\fa, \fb, \fc}^{\b,\g}; f)
$$
converges to $f$ in $L^2(\sW_{\fa, \fb, \fc}^{\b,\g},\Lambda_{\fa, \fb, \fc})$ and it is the best approximation to 
$f$ from $\Pi_n^{\circ,\sE}$ in the norm of $L^2(\sW_{\fa, \fb, \fc}^{\b,\g},\Lambda_{\fa, \fb, \fc})$. 

For $f \in L^p(\sW_{\fa, \fb, \fc}^{\b,\g},\Lambda_{\fa, \fb, \fc})$, $1 \le p < \infty$ and $f \in C(\Lambda_{\fa, \fb,\fc})$
if $p = \infty$, we define 
$$
   E_n(f)_{L^p\big(\sW_{\fa, \fb, \fc}^{\b,\g},\Lambda_{\fa, \fb, \fc}\big)} 
         = \inf_{g \in\Pi_n^{\circ,\sE}} \|f - g\|_{L^p\big(\sW_{\fa, \fb, \fc}^{\b,\g},\Lambda_{\fa, \fb, \fc}\big)},
$$
where the space becomes $C(\Lambda_{\fa,\fb,\fc})$ if $p = \infty$. This is the error of best approximation by polynomials 
and it can be characterized via the K-functional defined by, for $r \in \NN$ and $\rho > 0$,  
$$
\sK_r(f; \rho)_{p,\sW_{\fa,\fb,\fc}^{\b,\g}} = \inf_{g} \left\{ \|f - g\|_{L^p\big(\sW_{\fa, \fb, \fc}^{\b,\g},\Lambda_{\fa, \fb, \fc}\big)} + 
  \rho^r \left \|\big (\fD_{\fa,\fb,\fc}^{\b,\g}\big )^\f r 2 g \right\|_{L^p\big(\sW_{\fa, \fb, \fc}^{\b,\g},\Lambda_{\fa, \fb, \fc}\big)} \right \}, 
$$
where the infimum is taken over $g \in C^r(\Lambda_{\fa,\fb,\fc})$, the space of functions that have $r$-th continuous 
derivatives, and the fractional operator of $\fD_{\fa,\fb,\fc}^{\b,\g}$ is defined by
$$
  (\fD_{\fa,\fb,\fc}^{\b,\g})^{\f r 2} \proj_n^{\circ, \sE}(\sW_{\fa,\fb,\fc}^{\b,\g}; f) = -n (n+2 \b+2\g+2)^{\f12} 
      \proj_n^{\circ, \sE}(\sW_{\fa,\fb,\fc}^{\b,\g}; f)
$$
for all $n =1,2,\ldots$ by using the relation \eqref{eq:spectral}. 

\begin{thm}\label{main-thmV0}
Let $\b \ge 0$, $\g \ge - \f12$, and $f \in L^p(\sW_{\fa,\fb,\fc}, \Lambda_{\fa,\fb,\fc})$ if $1 \le p < \infty$, 
and $f \in C(\Lambda_{\fa,\fb,\fc})$ if $p = \infty$. Then, for $r \in \NN$ and $n =1,2,\ldots$, there hold 
\begin{enumerate} [\rm (i)]
\item {\it direct theorem}:
$$
    E_n(f)_{L^p\big(\sW_{\fa, \fb, \fc}^{\b,\g},\Lambda_{\fa, \fb, \fc}\big)} 
        \le c \, K_r(f;  n^{-1})_{p,\sW_{\fa, \fb, \fc}^{\b,\g}} ;
$$
\item {\it inverse theorem}:  
$$
   K_r(f; n^{-1})_{p,\sW_{\fa, \fb, \fc}^{\b,\g}} 
         \le c \,n^{-r} \sum_{k=0}^n (k+1)^{r-1} E_k(f)_{L^p\big(\sW_{\fa, \fb, \fc}^{\b,\g},\Lambda_{\fa, \fb, \fc}\big)}.
$$
\end{enumerate}
\end{thm}

\begin{proof}
This follows from the corresponding result on the unit ball established in \cite[Theorem 4.3]{X05}, for which the 
stated characterization holds for $E_n(f)_{L^p(\sW_\BB^\k, \BB^2)}$ with the corresponding K-functional defined via
the differential operator $\CD_\BB^\k$. For example, the direct theorem on $\BB^2$ is given by
$$
  E_n(f)_{L^p(\sW_\BB^\k,\BB^2)}:= \inf_{g \in\Pi_n^2} \|f - g\|_{L^p(\sW_\BB^\k,\BB^2)},
   \le c K_r(f; n^{-1})_{p, \sW_\BB^\k}
$$
where the $K$-functional on the disk is defined by 
$$
   K_r(f; \rho)_{p, \sW_\BB^\k} = \inf_{g\in C^r(\BB)} \left \{ \|f-g\|_{L^p(\sW_\BB^\k,\BB^2)} + 
       \rho^r \left \| (\CD_\BB^\k)^\frac r 2  g\right\|_{L^p(\sW_\BB^\k,\BB^2)} \right\}
$$
If $f$ on $\BB^2$ is even in its second variable, then we can choose $g$ to be even in its second variable in both 
of the above definitions, which shows, in particular, that we can replace $\BB^2$ by $\BB_+^2$ and $\Pi_n^2$ by
$\Pi_n^{\circ, \sE}$. Moreover, if $g$ is a polynomial that is even in its second variable, then $g\circ \psi^{-1}$ has 
the same symmetry since 
\begin{equation*} 
       \psi^{-1}(u,v) = \left(s, \sqrt{(\fb-\fa)t^2 + \fa(1-s^2) + \fc s^2}\right). 
\end{equation*}
Consequently, by Lemma \ref{lem:st-uv}, we conclude that 
$$
  E_n(f)_{L^p\big(\sW_{\fa, \fb, \fc}^{\b,\g},\Lambda_{\fa, \fb, \fc}\big)} 
     = E_n\left( f\circ \psi^{-1} \right)_{L^p\big(\sW_\BB^{\b,\g},\BB^2\big)}.
$$
Furthermore, applying \eqref{eq:fD=D} on $f\circ \psi^{-1}$, we obtain 
$\fD_{\fa,\fb,\fc}^\k g = \left(\CD_\BB^\k g\circ \psi^{-1} \right) \circ \psi$, which implies, together with the integral
identity in Lemma \ref{lem:st-uv}, that 
$$
   K_r(f, \rho)_{p,\sW_{\fa, \fb, \fc}^{\b,\g}} 
       = K_r\left( f\circ \psi^{-1}, \rho \right)_{p,\sW_\BB^{\b,\g}}.
$$
Consequently, both direct and inverse theorems on $\Lambda_{\fa, \fb, \fc}$ follow from the corresponding result
on the disk. 
\end{proof}
 
The proof of the theorem shows that polynomial approximation on the curved domain $\Lambda_{\fa, \fb, \fc}$
follows from the corresponding result for functions that are even on the unit disk. This applies to many other 
results, for example, on the near-best approximation operator, defined by a resampling of the Fourier partial sum,
that has a highly localized kernel \cite{DaiX, X21}. Moreover, the addition formula leads to a natural 
pseudo-convolution structure which can be utilized to study a variety of problems about the Fourier orthogonal 
expansions on $\Lambda_{\fa, \fb, \fc}$; for example, summabilities \cite{X21}, maximum functions and 
multipliers \cite{X22}.  
 
We end this section by specializing our setup to the special case of $\fa = 0$, $\fb = 1$, and $\fc \ge 0$. 
In this case, we  denote $\Lambda_{\fa,\fb,\fc}$ by 
$$
  \Lambda_{\fc} = \left \{(u,v): \sqrt{\fc} |u| \le v \le \sqrt{1-u^2 + \fc u^2}, \quad  |u| \le 1\right \}. 
$$
The weight funciton $\sW_{0, 1,\fc}^{\b,\g}$ on the domain is degenerated to 
\begin{align} \label{eq:sWc}
  \sW_{\fc}^{\b,\g}(u,v)  = v (v^2-\fc u^2)^{\b-\f12} \left(1-u^2-v^2+ \fc u^2 \right)^{\g}, \quad (u,v) \in \Lambda_\fc,
\end{align}
where $\b\ge -\f12$ and $\g > -1$. The simplest case is $\sW_\fc^{\f12,0} (u,v) = |v|$. The domains are depicted in Figure 2. 

It is worthwhile to remark that the domain $\Lambda_{\fc}$ is a triangle $\{(u,v): |u|\le v \le 1\}$ when $\fc = 1$. The 
classical Jacobi weight function on this triangle is $(v^2-u^2)^\b (1-v)^\g$, which is different from the
weight funciton $\sW_1^{\b,\g}(u,v) = v(v^2-u^2)^\b(1-v^2)^\g$ with $\fc =1$. Thus, our orthogonal structure is 
different from the classical one. For $\fc \ne 1$, the domain is a curved one depicted in Figure 2 and a 
circular sector of the unit disk if $\fc < 1$. We are not aware of any regular orthogonal basis of polynomials on
such a domain. It is enlightening that approximation and orthogonal structure on such a domain can be 
identified with those on the unit disk. 

By the result established in the previous subsection, OPs that consist of an orthogonal basis for 
$\CV_n(\Lambda_\fc, \sW_\fc^{\b,\g})$ are 
\begin{equation} \label{eq:sQ_OP}
   \sQ_{j,n}^{\b,\g}(u,v) =  C_{n-2j}^{(2 j + \b+\g+1)}(u) (1-u^2)^{j} P_j^{(\g, \b-\f12)}\bigg(2 \frac{-\fc u^2+v^2}{(1-u^2)} -1\bigg)
\end{equation}
with $0 \le j \le \lfloor \f n 2 \rfloor$. They are the eigenfunctions of the second-order differential operator
 $\fD_\fc^{\b,\g} = \fD_{0,1,\fc}^{\b,\g}$ given by 
\begin{align*}
 \fD_{\fc}^{\b,\g} \, & = (1-u^2) \partial_{uu} - 2(-\fc+v^2) \frac{u}{v} \partial_{u v} + \left(1 - v^2
       - \frac{ \fc (1-\fc) u^2}{v^2}\right)\partial_{vv}\\
  & + \left( \fc + \frac{ \fc (1-\fc) u^2}{v^2} +2\b \right) \frac{1}{v}\partial_{v} 
    + (2(\b + \g)+3) \left( u \partial_u + v\partial_{v}\right).
\end{align*}
Although the coefficients in front of the differentials in $\fD_{\fc}^{\b,\g}$ are rational functions, it does not
have a singularity for integrability, as seen by \eqref{eq:fD=D} and Lemma \ref{lem:st-uv}.

\section{Approximation and OPs on domains of revolution}
\setcounter{equation}{0}

We consider OPs on domains defined via a rotation around an axis. The simplest example of such a 
domain is the unit ball in $\RR^d$. OPs on the unit ball have been studied extensively; we review
relevant results that we shall need in the first subsection. Our goal is to consider OPs on a class of 
domains arising from the rotation of a fully symmetric domain, studied in \cite{X24}, which we recall 
in the second subsection. We show that these families of OPs can be related to OPs on the unit 
ball, as an extension of the results in the previous section, in the third section. 

\subsection{OPs on the unit ball} 
The classical weight function $W_\mu$ on the unit ball $\BB^d$ of $\RR^d$ is defined by 
$$
       W_\g (\xb) = (1-\|\xb\|^2)^{\g}, \quad \g > -1. 
$$
An orthogonal basis for $\CV_n(W_\g, \BB^d)$ can be given in terms of the Jacobi polynomials and 
spherical harmonics. The latter are the restrictions of homogeneous harmonic polynomials on the unit 
sphere $\sph$, and they are OPs on $\sph$. Let $\CH_n^d$ be the space of spherical harmonics of degree 
$n$ of $d$ variables. It is known that
\begin{equation} \label{eq:dimHnd}
 \dim \CH_n^d = \binom{n+d-1}{n} - \binom{n+d-3}{n-2}.
\end{equation}
Spherical harmonics of different degrees are orthogonal in $L^2(\sph)$. More precisely, let $\{Y_\ell^n\}$
be an orthonormal basis of $\CH_n^d$. Then, 
$$
   \frac{1}{\o_d} \int_{\sph} Y_\ell^n(\xi) Y_{\ell'}^{n'} (\xi) \d \s(\xi) =  \delta_{\ell,\ell'} \delta_{n,n'}, 
$$
where $\d \s$ denotes the surface measure $\d \s$ on $\sph$ and $\o_d$ denotes the surface area of $\sph$.
An orthogonal basis of $\CH_n^d$ can be given in terms of Jacobi polynomials in spherical coordinates. 

Let $\{Y_\ell^{n-2m}: 1 \le \ell \le  \dim \CH_{n-2m}\}$ be an orthonormal basis of $\CH_{n-2m}^d$ for $ 0\le m \le n/2$. 
Define \cite[(5.2.4)]{DX}
\begin{equation}\label{eq:basisBd}
  \Pb_{\ell, m}^n (\Wb_\g; \xb) = P_m^{(\g, n-2m+\f{d-2}{2})} \left(2\|\xb\|^2-1\right) Y_{\ell}^{n-2m}(\xb).
\end{equation}
Then $\{\Pb_{\ell,m}^n(\Wb_\g): 0 \le m \le n/2, 1 \le \ell \le \dim \CH_{n-2m}\}$ is an orthogonal basis of 
$\CV_n^d(W_\g,\BB^d)$. OPs in $\CV_n(W_\mu, \BB^d)$ are eigenfunctions of a second-order differential operator 
defined by 
$$
   \CD_\BB^\g: = \Delta - \la x ,\nabla \ra^2 - (2\g + d) \la x, \nabla\ra; 
$$
more precisely, (\cite[(5.23)]{DX}) 
$$
  \CD_\BB^\g \Pb = - n(n+2\g+d) \Pb, \quad \forall \Pb \in  \CV_n(W_\mu,\BB^d).
$$

For our purpose, we need to consider another class of OPs on the unit ball with respect to the
weight function 
\begin{equation} \label{eq:WbgBB}
  \Wb_\BB^{\b,\g}(\yb) = |y_{d+1}|^{2\b} (1-\|\yb\|^2)^\g, \qquad  \yb \in \BB^{d+1},
\end{equation}
which we state on the ball $\BB^{d+1}$ instead of $\BB^d$ in comparison with the notation in the next subsection. 
This is a special case of the weight function $h_\k(\yb)(1-\|\yb\|^2)^\g$ with $h_\k(\yb) = \prod_{i=1}^{1} |y_i|^{2\k_i}$,
invariant under the group $\ZZ_2^{d+1}$ \cite[Section 8.1]{DX}. Orthogonal basis for the weight function 
$h_\k^2(\xb)(1-\|\yb\|)^{d+1}$ can be deduced from the $h$-harmonics associated with the Dunkl Laplacian. 
Instead of going through the tedious reduction from the existing basis (cf. \cite[Proposition 8.1.5]{DX}), we provide a direct
verification of the result. 

\begin{prop}  \label{prop:OP_Qb}
Let $\Pb_{\ell,j}^n(\Wb_\g)$ be the OPs on $\BB^d$ in \eqref{eq:basisBd}. Let 
\begin{equation} \label{eq:OP_Qb}
  \Qb_{\ell, j, k}^n (\Wb_\BB^{\b,\g}; \xb,t) = C_{n-k}^{(k+\g+\frac{d+1}{2}, \b)}(t) (1-t^2)^\frac{k}{2} 
     \Pb_{\ell, j}^k \left(\Wb_\g; \frac{\xb}{\sqrt{1-t^2}}\right).
\end{equation}
Then, $\{\Qb_{\ell,j,k}^n(\Wb_\BB^{\b,\g}):  1 \le \ell \le \dim \CH_{n-2j}^d, 0 \le m \le k/2, 0\le k \le n\}$ consists of an 
orthogonal basis for the space $\CV_n(\Wb_\BB^{\b,\g}, \BB^{d+1})$.
\end{prop}

\begin{proof}
Let $\la \cdot,\cdot\ra_{\b,\g}$ be the inner product in $L^2(\Wb_{\b,\g}, \BB^{d+1})$. It suffices to identify 
a polynomial $p_{n-k}$ of one variable such that 
$$
\Qb_k^n(\xb,t) = p_{n-k}(t) (1-t^2)^\frac{k}{2} \Pb_k \left(\frac{\xb}{\sqrt{1-t^2}}\right), \quad \Pb_k = \Pb_{\ell, m}^k(\Wb_\g), 
$$
satisfies $\la \Qb_k^n, \Qb_{k'}^{n'} \ra_{\b,\g} = 0$ if $(k,n) \ne (k',n')$. Indeed, since $\Pb_{\ell, m}^k(\Wb_\g)$
is an even polynomial if $k$ is even and an odd polynomial if $k$ is odd, it follows that $\Qb_k^n$ is a
polynomial of degree at most $n$. Moreover, the proof below shows that the polynomials containing 
$\Pb_{\ell,m}^k(\Wb_\g)$ of the same degree $k$ are orthogonal. Now, for $(\xb,t) \in \BB^{d+1}$, a change of 
variables $\xb \to \sqrt{1-t^2}\, \ub$, $\ub \in \BB^d$, leads to
\begin{align*}
\int_{\BB^{d+1}}  f(\xb,t)  \d x \d t =  \int_{-1}^1  \int_{\BB^d} f\left(\sqrt{1-t^2} \ub, t\right)  \d \ub  (1-t^2)^{\f d 2} \d t
\end{align*}
and $\Wb_{\b,\g}(\xb,t) =|t|^{2\b} (1-\|\ub|^2)^\g  (1-t^2)^\g$. It follows readily then that 
\begin{align*}
 \la \Qb_k^n, \Qb_{k',n'}^{\b,\g}\ra_{\b,\g} \,& = 
   \int_{-1}^1 p_{n-k}(t) p_{n'-k'}(t) |t|^{2\b} (1-t^2)^{\g+k+\f{d}{2}} \d t \int_{\BB^d} \Pb_k(\ub) \Pb_{k'} (\ub)
     \Wb_{\g} (\ub) \d u \\
  & =  \delta_{k, k'} \int_{-1}^1 p_{n-k}(t) p_{n'-k}(t) |t|^{2\b} (1-t^2)^{\g+\f{d}{2}} \d t,  
\end{align*}
so that the orthogonality holds with $p_{n-k}$ being a generalized Gegenbauer polynomial given by 
$p_{n-k} = C_{n-k}^{(k+\g+\frac{d+1}{2}, \b)}$. This completes the proof.
\end{proof}
  
 OPs in the space $\CV_n(\Wb_\BB^{\b,\g}, \BB^{d+1})$ are eigenfunctions of the differential-difference operator
 $\CD_\BB^{\b,\g}$ defined by, for $\yb \in \BB^{d+1}$, \cite[Theorem 8.1.3 and (7.5.3)]{DX}
\begin{align}\label{eq:D_BBd+1}
  \CD_\BB^{\b,\g} =  \Delta - \la \yb ,\nabla \ra^2 - (2\b + 2\g + d+1) \la \yb, \nabla\ra 
     + \b \left(\frac{2}{y_{d+1}} \partial_{d+1} - \frac{1-\s_{d+1}^2}{y_{d+1}^2} \right) 
\end{align}
where $\s_{\d+1}$ denotes the reflection operator in $y_{d+1}$, as in \eqref{eq:diffB2}. More precisely, 
$$
  \CD_\BB^{\b,\g} \Qb =  - n(n + 2\b + 2\g +d+1) \Qb, \quad \forall \Qb \in  \CV_n(\Wb_{\b,\g},\BB^{d+1}).
$$
Just like in the case of two variables, the operator $\CD_\BB^{\b,\g}$ becomes a differential operator when acting on
functions that are even in the $(d+1)$-th variable. 

Let $\Pb_n(\Wb_\BB^{\b,\g}; \cdot,\cdot)$ denote the reproducing kernel of the space $\CV_n(\Wb_\BB^{\b,\g},\BB^{d+1})$,
which satisfies 
$$
   \Pb_n\left(\Wb_\BB^{\b,\g}; \xb,\yb\right) = \sum_{|\alpha|=n} \Qb_\alpha(\xb) \Qb_\alpha(\yb), \qquad \xb, \yb \in \BB^{d+1}, 
$$
where $\{\Qb_\alpha:  |\alpha| = n, \, \a \in \NN_0^{d+1}\}$ is an orthonormal basis of  $\CV_n(\Wb_\BB^{\b,\g},\BB^{d+1})$.
This kernel has a closed-form formula, known as an addition formula for OPs: for $\xb=(\xb',t)$ and $\yb = (\yb',s)$,  
\begin{align}\label{eq:kernel_Bd}
    \Pb_n\left(\Wb_\BB^{\b,\g}; \xb,\yb\right) = c_{\b,\g} \int_{-1}^1 \int_{-1}^1 
       & Z_n^{\b+\g+\frac{d+1}{2}} \left(\la \xb',\yb' \ra + u\, t s + v \sqrt{1-\|\xb\|^2}  \sqrt{1-\|\yb\|^2} \right) \notag \\
        & \times  (1+u) (1-u^2)^{\b-1} (1-v^2)^{\g-\f12} \d u \d v,
\end{align}
where $c_{\b,\g}$ is the normalization constant, $Z_n^\l$ is defined in \eqref{eq:Zn}, and the identity holds under 
the limit if $\b = 0$ and/or $\g =-\f12$; see \cite[Theorem 8.1.16]{DX}.  

\subsection{OPs for a family of domains of revolution}
We consider domains derived from rotating a fully symmetric domain in $\RR^2$. Let $\Omega$ be a 
domain of $\RR^2$ that is symmetric in its first variable; that is, $(s,t) \in \Omega$ implies $(-s,t) \in \Omega$. Let 
$\Omega_+ = \{(s,t) \in \Omega: s \ge 0\}$. We consider the domain defined by 
$$
 \XX^{d+1} = \left\{ (\xb,t) \in  \RR^{d+1}, \,\,  \xb \in \RR^d, \,\, , t \in\RR, \,\, (\|\xb \|,t ) \in \Omega_+ \right \},
$$
which is the rotation of $\Omega_+$ around the $t$ axis for $d=2$. Let $\sW$ be a weight function defined 
on $\Omega_+$. On $\XX^{d+1}$ we define
$$
      \Wb(\xb) = \sW(\|x\|,t), \qquad (\xb, t) \in \XX^{d+1}.
$$
Construction of orthogonal basis in $L^2(\Wb,\XX^{d+1})$ has been studied in \cite{X24}, motivated by the recent
advances about OPs on circular quadratic domains \cite{OX, X20, X21a}. In the present section, we 
revisit the case when $\Omega$ is fully symmetric. 

Let $\CV_n(\Wb, \XX^{d+1})$ denote the space of OPs of degree $n$. In the case that $\Omega$ is fully
symmetric, orthogonal bases have been constructed for $\CV_n(\Wb, \XX^{d+1})$ in \cite{X24} and, just 
as in the case of two variables studied in the previous section, the bases are split for those being even 
in the $t$-variable and those being odd in the $t$-variable. We denote by 
$$
      \CV_n^{\sE}(\Wb, \XX^{d+1}) = \left\{Y \in \CV_n(\Wb, \XX^{d+1}): Y(\xb, - t) = Y(\xb,t) \right\}
$$ 
the subspaces of $\CV_n(\Wb, \XX^{d+1})$ that consisit of OPs even in the $t$-variable. 
Let $\sW(s,t) = \sw(s^2,t^2)$ be the fully symmetric weight on $\Omega$, where $\sw$ is a 
weight function on $\sqrt{\Omega}$, the domain defined in \eqref{eq:sqrtOmega}. For $k \in \NN$, define 
\begin{equation} \label{eq:Wk_fullsym}
    \sw_{-\f12, -\f12}^{(k)}(s,t) = s^{k+\frac{d-2}{2}} t^{-\f12} \sw(s,t), \quad (s,t) \in \sqrt{\Omega}. 
\end{equation}

\begin{prop} \label{prop:OP_fullSym}
Let $\big\{\sQ_{j,m}\big(\sw^{(k)}_{-\f12, -\f12}\big): 0 \le j \le m\big\}$ be an orthogonal basis of the
space $\CV_m\big(\sw^{(k)}_{- \f12, -\f12}, \sqrt{\Omega}\big)$ and let $\{Y_\ell^k: 1 \le \ell \le \dim \CH_k^d\}$ be 
an orthogonal basis for $\CH_k^d$. Then the space $\CV_n^\sE\big(\Wb,\XX^{d+1}\big)$ has an 
orthogonal basis given by 
\begin{equation} \label{eq:basisE}
 \Qb_{j,n-2m,\ell}^n(\xb,t) =  \sQ_{j,m} \Big(\sw_{-\f12, -\f12}^{(n-2m)}; \|\xb\|^2, t^2\Big) Y_\ell^{n-2m} (\xb)
\end{equation}
for $1 \le \ell \le \dim \CH_{n-2m}^d, \,0\le j \le  m \le \lfloor \frac{n}{2} \rfloor$.
\end{prop}

This is established in \cite[Proposition 4.5]{X24}. It follows that 
\begin{align*}
  \dim \CV_n^\sE(\Wb,\XX^{d+1}) = 
     \sum_{m=0}^{\lfloor \frac{n}{2} \rfloor} (m+1) \dim \CH_{n-2m}^d = 
        \sum_{m=0}^{\lfloor \frac{n}{2} \rfloor} \binom{n-2m+d-1}{d-1}. 
\end{align*}

We are interested in the specific case when $ \triangle_{\fa,\fb,\fc}= \sqrt{\Omega_{\fa,\fb,\fc}}$, given in 
the previous section and $\sw(s,t) = t^{\f12} \sw_{\fa,\fb,\fc}^{\k_1,\k_2-\f12, \k_3}(s,t)$, using the notation of \eqref{eq:w_abc}. In this
setting, the domain becomes 
$$
\XX_{\fa,\fb,\fc}^{d+1} := \left \{ (\xb,t) \in  \RR^{d+1}: \,\,  \xb \in \RR^d, \,\,  t \in\RR,\,\,  
     (\|\xb \|,t ) \in \Omega_{\fa,\fb,\fc}^+ \right\},
$$
where $ \Omega_{\fa,\fb,\fc}^+ = \{(u,v) \in  \Omega_{\fa,\fb,\fc}: u \ge 0\}$, and the corresponding 
weight function $\Wb$ becomes 
\begin{align*}
  \Wb_{\fa,\fb,\fc}^\k (\xb,t) \,& = |t| \sw_{\fa,\fb,\fc}^{\k_1,\k_2-\f12, \k_3} \big(\|\xb\|^2, t^2 \big). 
\end{align*}
With these choices, the weight function $\sw_{-\f12,\pm \f12}^{(k)}$ becomes 
$$
    \sw_{-\f12, -\f12}^{(k)}(s,t) = s^{k+\frac{d-2}{2}} \sw_{\fa,\fb,\fc}^{\k_1,\k_2-\f12, \k_3}(s,t) 
      =  \sw_{\fa,\fb,\fc}^{\k_1+k+\frac{d-2}{2},\k_2-\f12, \k_3}(s,t), 
$$
which is a classical Jacobi weight on the triangle. Hence, by \eqref{eq:OPabc}, an orthogonal
basis of the space $\CV_m\big(\sw^{(k)}_{- \f12, -\f12}, \sqrt{\Omega}\big)$ is given by the Jacobi 
polynomials on the triangle,
$$
  \sQ_{j,m}\left(\sw^{(k)}_{-\f12, - \f12}; s,t \right) = \sT_{j,m}^{\k_1+k+\frac{d-2}{2},\k_2-\f12, \k_3}
       \left(s, \frac{-\fa+(\fa-\fc)s +t}{\fb-\fa} \right). 
$$
Consequently, by Proposition \ref{prop:OP_fullSym} and \eqref{eq:basisE}, an orthogonal basis for
$\VV_n^\sE(\Wb_{\fa,\fb,\fc}^{\b,\g}, \XX^{d+1})$ is given by 
\begin{equation}\label{eq:OP-k-abc}
  \Qb_{j,m,\ell}^n  (\xb,t) = 
 \sT_{j,m}^{\k_1+n-2m +\frac{d-2}{2},\k_2 -\f12, \k_3} \left(\|\xb\|^2, \frac{-\fa+(\fa-\fc)\|\xb\|^2 +t^2}{\fb-\fa} \right) Y_\ell^{n-2m} (\xb)
\end{equation}
with $1 \le \ell \le \dim \CH_{n-2m}^d, \,0\le j \le  m \le \lfloor \frac{n}{2} \rfloor$.  

As in the case of two variables, we are particularly interested in the upper half of the domain $\XX_{\fa,\fb,\fc}^{d+1}$,
which we denote by
$$
    \Lambda_{\fa,\fb,\fc}^{d+1} := \left\{ (x,t) \in \XX_{\fa,\fb,\fc}^{d+1}: t \ge 0 \right\}
$$
and the case $\k = (0, \b, \g)$, so that the weight function becomes 
\begin{align*}
  \Wb_{\fa,\fb,\fc}^{\b,\g}(\xb,t)  \, &= |t| \sw_{\fa,\fb,\fc}^{0,\b-\f12, \g} \big(\|\xb\|^2, t^2 \big) \\ 
  &   = |t| \left(\frac{-\fa+(\fa-\fc)\|\xb\|^2+t^2}{\fb-\fa}\right)^{\b-\f12} \left(\frac{\fb-(\fb-\fc)\|\xb\|^2-t^2}{\fb-\fa}\right)^{\g}. 
\end{align*}
For $d = 3$, the domain $\Lambda_{\fa,\fb,\fc}^{d+1}$ is the rotation of fully symmetric domain 
$\Lambda_{\fa,\fb,\fc}$ in the plane around the vertical axis. For $\fa = 0$, we depict them in Figure 4,  
\begin{figure}[htb]
\begin{minipage}{1\textwidth}
\centering
\includegraphics[width=4cm]{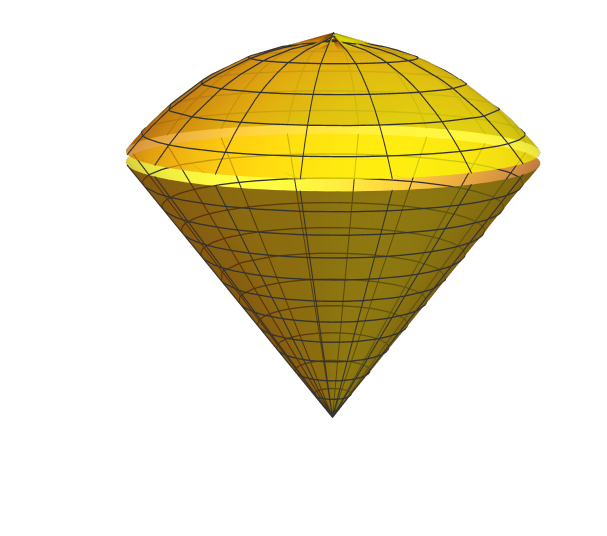}  \quad 
 \includegraphics[width=4.2cm]{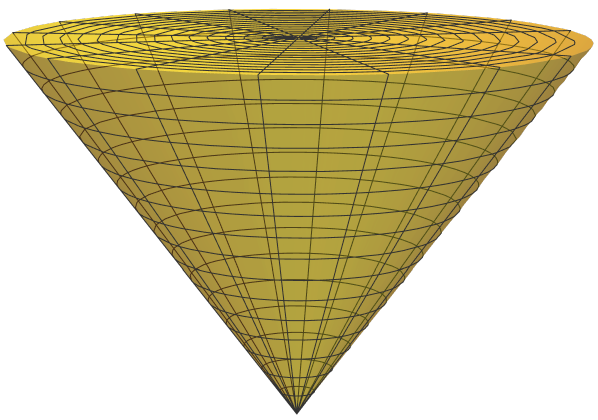} \quad
 \includegraphics[width=3.8cm]{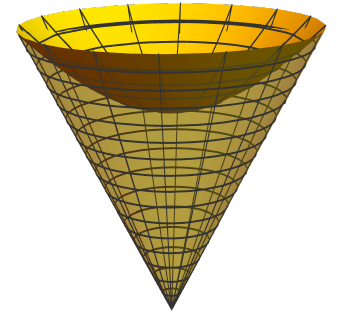}
\end{minipage}
\caption{Left $\fb > \fc$ \qquad Center $\fc = \fb$ \qquad Right $\fb < \fc$ \qquad}
\label{Fig4}
\end{figure}
where the left and the right figures are rotations of the two domains depicted in Figure 3, and the one 
in the center is a solid cone. For $\fa> 0$, we also have three types of domains that are of rotations of
the domains depicted in Figure 2. 

We have an analog of Lemma \ref{lem:st-uv}. Let $\BB_+^{d+1} = \{(\xb,t) \in \BB^{d+1}: t \ge 0\}$ be the upper part of 
the unit ball $\BB^{d+1}$. Recall $\ft(u,v)$ defined in \eqref{eq:tuv} and $\Wb_\BB^{\b,\g}$ defined in \eqref{eq:WbgBB}. 

\begin{lem} \label{lem:st-uv2}
A bijection bewteen $\Lambda_{\fa,\fb,\fc}^{d+1}$ and $\BB_+^{d+1}$ is given by 
\begin{equation}\label{st-uv2}
\psi: (\xb, t) \in  \Lambda_{\fa,\fb,\fc}^{d+1} \mapsto (\ub, v) \in \BB_+^2, \quad \psi(\xb,t) = (\xb,  \ft (\|\xb\|, t)) = (\ub,v),
\end{equation}
which leads to the integral identity 
$$
   \int_{\Lambda_{\fa,\fb,\fc}^{d+1}} f(\xb,t) \Wb_{\fa,\fb,\fc}^{\b,\g} (\xb,t) \d \xb \d t = 
       \int_{\BB_+^{d+1}} (f\circ\psi^{-1})(\ub,v) \Wb_\BB^{\b,\g}(\ub, v) \d \ub \d v. 
$$
\end{lem}

\begin{proof}
The proof follows along the line of Lemma \ref{lem:st-uv}. Computing the Jacobian of the change of variables 
$(\ub,s) \mapsto (\vb,t)$, it follows readily that 
$$
s \, \d \ub \d s = (\fb - \fa)t\, \d \vb \d t, 
$$
which leads to the identity 
$$
      \Wb_{\fa,\fb,\fc}^{\k} (\ub,s) \d \ub \d s =  t^{2\b} \left(1-\|\vb\|^2-t^2\right)^{\k_2} \d \vb \d t= \Wb_{\b,\g}(\vb,t) \d \vb \d t,
$$
so that the integral identity follows from a change of variables. 
\end{proof}

As an extension of Theorem \ref{prop:sQjn}, OPs in the space 
$\CV_n^\sE\big(\Wb_{\fa,\fb,\fc}^{\b,\g}, \Lambda_{\fa,\fb,\fc}^{d+1}\big)$ can be given in terms of the 
semi-classical (classical if $\b =0$) OPs in $\CV_n\big(\Wb_\BB^{\b,\g}, \BB^{d+1})$ on the unit ball. 

\begin{thm}\label{thm:QljnX}
Let $\Qb_{\ell, j, k}^n(\Wb_{\fa,\fb,\fc}^{\b,\g})$ be the polynomials defined by
\begin{align} \label{Q=Q}
 \Qb_{\ell, j, k}^n\big(\Wb_{\fa,\fb,\fc}^{\b,\g}; \xb,t \big)
 & =  
   C_{n-k}^{(k + \frac{d+1}2 +\g,\b)}(\ft) (1-\ft^2)^{\frac{k}{2}} 
         \Pb_{\ell,j}^k \Big(\Wb_\BB^{\b,\g}; \frac {\xb} {\sqrt{1-\ft^2} } \Big) \\
 &  = \left(\Qb_{\ell, j, k}^n \big(\Wb_\BB^{\b,\g}\big) \circ \psi \right) (\xb,t), \notag
\end{align}
where $\ft = \ft (\|\ub\|, t)$ and let $\Qb_{\ell, j, k}^n \big(\Wb_\BB^{\b,\g}\big)$ be defined in \eqref{eq:OP_Qb}. 
Then, the set 
$$
\left \{ \Qb_{\ell,j,k}^n\big(\Wb_{\fa,\fb,\fc}^{\b,\g}\big):  1 \le \ell \le \dim \CH_{n-2j}^d, 0 \le j \le k/2, 0\le k \le n \right\}
$$ 
consists of an orthogonal basis for the space $\CV_n^\sE\big(\Wb_{\fa,\fb,\fc}^{\b,\g}, \Lambda_{\fa,\fb,\fc}^{d+1}\big)$. 
\end{thm}

\begin{proof}
By \eqref{eq:OP-k-abc} and \eqref{eq:triOP}, as well as \eqref{eq:gGegen}, a basis of 
$\VV_n^\sE(\Wb_{\fa,\fb,\fc}^{\b,\g}, \Lambda_{\fa,\fb,\fc}^{d+1})$ consists of 
\begin{align*}
   & \sT_{j,m}^{n-2m +\frac{d-2}{2},\b -\f12, \g} \left(\|\xb\|^2, \ft(\xb,t)^2 \right) Y_\ell^{n-2m} (\xb) \\
  & = \mathrm{const.} C_{2m-2j}^{2j+n-2m+\frac{d+1}{2}+\g,\b-\f12}(\ft) (1-\ft^2)^j P_j^{(\g, n-2m +\frac{d-2}{2})}
  \left (\frac{2\|\ub\|^2}{1-\ft^2} -1\right)Y_\ell^{n-2m} (\xb), 
\end{align*}  
where $\ft = \ft(\|\xb|,t)$ is used in the second line. Let $k$ be defined by $n-2m = k-2j$. Then the last expression
becomes, by the homogeneity of $Y_\ell^k$,  
\begin{align*}
 C_{n-k}^{(k + \frac{d+1} +\g,\b)}& (\ft)  (1-\ft^2)^{\frac{k}{2}}
     P_j^{(\g, k-2j+\frac{d-2}{2})}\left (\frac{2\|\xb\|^2}{1-\ft^2} -1\right)Y_\ell^{k-2j} \left(\frac{\xb}{\sqrt{1-\ft^2}}\right)  \\
  & =       C_{n-k}^{(k + \frac{d+1} +\g,\b)}(\ft) (1-\ft^2)^{\frac{k}{2}} \Pb_{\ell,j}^k \left(\Wb_\g; \frac{\xb}{\sqrt{1-\ft^2}}\right)
  = \Qb_{\ell, j,k}^n (\Wb_\BB^{\b,\g}; \xb, \ft),
\end{align*} 
where the first step follows by \eqref{eq:basisBd} and the second one by \eqref{eq:OP_Qb}. Recall $\ft = \ft(\|\xb\|,t)$. 
Applying \eqref{st-uv2} shows that $ \Qb_{\ell, j, k}^n \big(\Wb_\BB^{\b,\g}\big) \circ \psi$ consist of a
desired basis. 
\end{proof}

For $d =1$, this theorem reduces to Theorem \ref{prop:sQjn}. If $(\fa, \fb, \fc) = (0, 1, 0)$, then 
$\Lambda_{\fa,\fb,\fc}^{d+1} = \BB_+^{d+1}$ and $\Wb_{\fa,\fb,\fc}^{\b,\g} = \Wb_{\BB}^{\b,\g}$. Another 
special case is $(\fa, \fb,\fc) = (0,1,1)$, for which the domain $\Lambda_{\fa,\fb,\fc}^{d+1}$ becomes a circular cone
and was studied in \cite{X21a}, where the phenomenon that OPs even in $t$-variable are different from OPs odd in
$t$-variable could be completely different was first observed, and this case was further extended in \cite{X24} to more
general $\fa,\fb,\fc$. The relation in Theorem \ref{thm:QljnX} is new, which shows, in particular, OPs on 
$\Lambda_{\fa,\fb,\fc}^{d+1}$ are equivalent to those on the upper part of the unit ball, or OPs on the unit ball that are 
even in the $(d+1)$-th variable.  

As a consequence of Theorem \ref{thm:QljnX}, we obtain an analog of Theorem \ref{thm:PDE_d=2} that shows
OPs $\Qb_{\ell, j, k}^n\big(\Wb_{\fa,\fb,\fc}^{\b,\g} \big)$ are eigenfunctions of a second
order differential operator, which we give explicitly in the next theorem. 

\begin{thm} \label{thm:PDE_d}
For $\b \ge 0$ and $\g > -1$, the polynomials in $\CV_n^{\circ, \sE}(\Wb_{\fa,\fb,\fc}^{\b,\g}, \Lambda_{\fa,\fb,\fc}^{d+1})$ 
are eigenfunctions of a differential operator, 
\begin{equation} \label{eq:spectral-d}
    \fD_{\fa,\fb,\fc}^{\b,\g} Y = -n(n+ 2\b+2\g+ d+1) Y, \qquad 
         Y \in \CV_n^{\circ, \sE}\left(\Wb_{\fa,\fb,\fc}^{\b,\g}, \Lambda_{\fa,\fb,\fc}^{d+1}\right),
\end{equation} 
and the operator is given explicitly by, for $(\xb, t) \in \Lambda_{\fa,\fb,\fc}^{d+1}$, 
\begin{align*}
 \fD_{\fa,\fb,\fc}^{\b,\g} \, & = \Delta_\xb - \la \xb,\nabla_\xb\ra^2  + 
    \left [(\fa-\fc) (\fb-\fc)\|\xb\|^2 +(\fa - t^2)(-\fb+ t^2)\right] \frac{1}{t^2} \partial_t^2 \\
  & + 2\left (\frac{\fc}{t}- t\right) \partial_t  \la \xb,\nabla_\xb\ra
         -(2\b+2\g+d+1) \big(\la \xb,\nabla_\xb\ra + t \partial_t \big) -v \partial_t \\
   & +(2 \fb \b + 2 \fa \g + \fa + \fc d) \frac{1}{t} \partial_t
   + \left [\fa \fb + (\fa - \fc) (-\fb + \fc ) ( u1^2 + u2^2)\right] \frac{1}{t^3} \partial_t. 
 \end{align*}
Moreover, in view of the operator $\CD_\BB^{\b,\g}$ on the unit ball in \eqref{eq:D_BBd+1},  
\begin{equation} \label{eq:fD=Dd}
        \fD_{\fa,\fb,\fc}^{\b,\g} (f\circ \psi) = \left(\CD_\BB^{\b,\g} f \right) \circ \psi.   
\end{equation}
\end{thm} 

\begin{proof}
The proof follows in line with the one for Theorem \ref{thm:QljnX}, which shows that we need to find a 
$\fD_{\fa,\fb,\fc}^\k$ that satisfies \eqref{eq:fD=Dd}. Let $\yb = (\ub, v) \in \BB_+^{d+1}$ with $\ub \in \BB^d$ and
$0\le v\le 1$. For $\yb \in \RR^{d+1}$, let $\Delta_d$ and $\nabla_d$ denote the operator acting on the first $d$ variable
and $\partial_{d+1}$ denote the partial derivative on the $d+1$ variable. Then, the operator $\CD_\BB^{\b,\g}$ can 
be written as 
\begin{align}\label{eq:CD_BBd}
\CD_\BB^{\b,\g} =  \Delta_d \, & - \la \ub, \nabla_d \ra^2 -2 v \partial_{d+1}  \la \ub, \nabla_d\ra
 - v^2 \partial_{d+1}^2 - v \partial_{d+1}\\
  & - (2 \beta +2 \gamma +d+1) \big( \la \ub, \nabla_d\ra + v \partial_{d+1}\big) + 2 \beta \frac1{v} \partial_{d+1}.  \notag
\end{align}
Let $F(\xb,t) = (f\circ \psi)(\xb,t) = f(\xb, \ft(\|\xb\|,t))$.  
Then, computing the derivatives of $F(\xb,t)$ and sovling backwards for the deirvatives of $f$ evaluated at $(\ub,v)$, 
we obtain, for example, 
\begin{align*}
  \Delta_d f\big(\ub, \ft(\|\ub\|, v)\big) \,& =  \Big [ \Delta_\ub - 2 (\fa-\fc)\frac1 v \partial_v \la \ub, \nabla_\ub \ra
  + (\fa-\fc)^2\|\ub\|^2 \frac{1}{v^2} \partial_v^2 \\
 \, & \quad  - \left ((\fa - \fc)\|\ub\|^2 + d v^2\right)\frac1 {v^3}  \partial_v \Big] F(\ub,v),  \\
 \la \ub, \nabla_d \ra f\big(\ub, \ft(\|\ub\|, v)\big) \,& = \Big [ \la \ub, \nabla_\ub\ra   -(a-c) \|\ub\|^2\frac{1}{v} \partial_v \Big]F(\ub,v), \\
\end{align*}
and, more involved,  
\begin{align*} 
  \partial_{d+1} \la \ub, \nabla_d\ra f\big(\ub, \ft(\|\ub\|, v)\big)
  \,&  = \bigg[ (\fb-\fa) \ft(\|\ub\|, v)\frac1 v \partial_v \la \ub, \nabla_\ub \ra  \\
      & \quad  - (\fa-\fc)(\fb-\fa) \ft(\|\ub\|, v) \|\ub\|^2 \left( \frac{1}{v^2} \partial_v^2 +\frac{1}{v^3} \partial_v \right) \bigg] F(\ub,v), \\
  \la \ub, \nabla_d \ra^2 f\big(\ub, \ft(\|\ub\|, v)\big) \,& =  \bigg[ \la \ub, \nabla_\ub \ra^2   
  - (\fa-\fc)\|\ub\|^2 \Big( \frac1 v \partial_v \la \ub, \nabla_\ub \ra +  \frac{\|\ub\|^2}{v^2} \partial_v^2 \Big) \\
  & \quad  -  (\fa-\fc)\|\ub\|^2 \left [ (\fa - \fc)\|\ub\|^2 + 2 v^2\right ]
   \frac1 {v^3}  \partial_v  \bigg] F(\ub,v).
\end{align*}
These computations are straightforward but tedious. Using them and other likewise identities on the right-hand side of \eqref{eq:CD_BBd}, we obtain, after a careful simplification, the stated result.
\end{proof} 

Let $\Pb_n^\sE(\Wb_{\fa,\fb,\fc}^{\b,\g})$ be the reproducing kernel of the space 
$\CV_n^\sE(\Wb_{\fa,\fb,\fc}^{\b,\g}, \Lambda_{\fa,\fb,\fc}^{d+1})$. By Lemma \ref{lem:st-uv2} and Theorem \ref{thm:QljnX},
it is easy to see that 
\begin{equation}\label{eq:ker=kerB}
   \Pb_n^\sE(\Wb_{\fa,\fb,\fc}^{\b,\g})\big((\xb,t), (\yb,s)\big) 
      =  \Pb_n^\sE\left(\Wb_\BB^{\b,\g}; \big( (\xb,\ft(\|\xb,r)), (\yb,\ft(\|\yb\|,s))\big)\right).  
\end{equation} 
A closed-form formula for this kernel follows immediately from the above relation and \eqref{eq:kernel_Bd}. We state
only the case $\b =0$ for simplicity. 

\begin{prop}
For $\g > -1$, $(\xb, t)$ and $(\yb,s)$ in $\Lambda_{\fa,\fb,\fc}^{d+1}$, 
\begin{align}\label{eq:kernel_d}
    \Pb_n^\sE \left(\Wb_{\fa,\fb,\fc}^{0,\g}; (\xb,t), (\yb,s) \right) = c_\g \int_{-1}^1  
       & Z_n^{\g+\frac{d+1}{2}} \big (\xi(\xb,t, \yb,s; v)\big) (1-v^2)^{\g-\f12} \d v,
\end{align}
where 
$$
 \xi(\xb,t, \yb,s; v) = \la X, Y \ra + v \sqrt{1-\|\xb\|^2 - \ft(\|\xb\|, t)^2}  \sqrt{1-\|\yb\|^2-\ft(\|\yb\|,s)^2}. 
$$
\end{prop}

We state the special case $(\fa,\fb,\fc) = (0,1,1)$ as an example for better reference. 
\begin{exam} \label{ex:cone}
 The circular cone is the domain 
$$
   \VV^{d+1}=  \Lambda_{0,1,1}^{d+1} = \{(\xb,t): \|\xb\| \le t, \quad 0 \le t \le 1\} 
$$
equipped with the weight function  
$$
   \Wb^{\b,\g}(\xb,t) = \Wb^{\b,\g}_{0,1,1} (\xb,t) = |t| (t^2-\|x\|^2)^{\b-\f12} (1-t)^\g.
$$
This case was studied in \cite{X21a}, where the parameters are $(\mu, \g)$, which corresponds to 
$(\b, \g-\f12)$ in our notation. The differential operator in Theorem \ref{thm:PDE_d} is reduced to 
\begin{align*}
 \fD_{0,1,1}^{\b,\g}  = \Delta_\xb\, & - \la \xb,\nabla_\xb\ra^2  + (1- t^2) \partial_t^2 
         +(1- t^2) \frac{2}{t} \partial_t  \la \xb,\nabla_\xb\ra \\   
     &   -(2\b+2\g+d+1) \big(\la \xb,\nabla_\xb\ra + t \partial_t \big) -v \partial_t  +(2 \fb + d) \frac{1}{t} \partial_t,
\end{align*}
which is Theorem 4.6 in \cite{X21a}. In this case, $\ft(\|\xb\|,t) = \sqrt{t^2 - \|\xb\|^2}$ and the addition formula 
\eqref{eq:kernel_d} is derived in \cite[Theorem 5.5]{X21a}.
\end{exam}
 
\subsection{Approximation on $\Lambda_{\fa,\fb,\fc}^{d+1}$}
As in the case of $d =2$ discussed in the previous section, much of the analysis on the domain
$\Lambda_{\fa,\fb,\fc}^{d+1}$ can be deduced from the corresponding results on the unit disk.

The orthogonal projection operator 
$$
 \proj_n^{\circ, \sE}\left( \Wb_{\fa,\fb,\fc}^{\b,\g}\right):  L^2\left(\Wb_{\fa,\fb,\fc}^{\b,\g}, \Lambda_{\fa,\fb,\fc}^{d+1}\right)
 \to  \CV_n^\sE\left(\Wb_{\fa,\fb,\fc}^{\b,\g}, \Lambda_{\fa,\fb,\fc}^{d+1}\right)
$$
is an integral operator, for $X = (\xb,t)$ and $Y =(\yb,s)$ in $ \Lambda_{\fa,\fb,\fc}^{d+1}$,
$$
 \proj_n^{\circ, \sE}\left( \Wb_{\fa,\fb,\fc}^{\b,\g}; X \right) =
 c_{\b,\g} \int_{\Lambda_{\fa,\fb,\fc}} f(Y) \Pb_n^\sE \left(\Wb_{\fa,\fb,\fc}^{\b,\g}; X, Y\right)
 \Wb_{\fa,\fb,\fc}^{\b,\g}(Y) \d Y.
$$
Extending $f$ on $\Lambda_{\fa,\fb,\fc}^{d+1}$ to the domain $\XX_{\fa,\fb,\fc}^{d+1}$ evenly by
defining $f(\xb, -t) = f(\xb,t)$, we obtain an analog of Theorem \ref{thm:Fourier}. 

\begin{thm} \label{thm:Fourier-d}
Let $\b\ge 0$ and $\g > -1$. If $f \in  L^2\big(\Wb_{\fa,\fb,\fc}^{\b,\g}, \Lambda_{\fa,\fb,\fc}^{d+1}\big)$, then 
\begin{equation}\label{eq:Fourier-d}
        f = \sum_{n=0}^\infty \proj_n^{\circ, \sE}\left( \Wb_{\fa,\fb,\fc}^{\b,\g}, f\right). 
\end{equation}
\end{thm}

Let $\Pi_n^{\circ, \sE}$ denote the space of polynomials at most $n$ in $d+1$ variables that are even in 
the last variable. Then 
$$
   \Pi_n^{\circ,\sE} = \bigoplus_{k=0}^n \CV_k\left(\Wb_{\fa, \fb, \fc}^{\b,\g},\Lambda_{\fa, \fb, \fc}^{d+1}\right). 
$$
For $f \in L^p(\Wb_{\fa, \fb, \fc}^{\b,\g},\Lambda_{\fa, \fb, \fc}^{d+1})$, $1 \le p < \infty$ and 
$f \in C(\Lambda_{\fa, \fb,\fc}^{d+1})$ if $p = \infty$, we define  
$$
   E_n(f)_{L^p\big(\sW_{\fa, \fb, \fc}^{\b,\g},\Lambda_{\fa, \fb, \fc}^{d+1}\big)}
       = \inf_{g \in\Pi_n^{\circ,\sE}} \|f - g\|_{L^p\big(\sW_{\fa, \fb, \fc}^{\b,\g},\Lambda_{\fa, \fb, \fc}^{d+1}\big)},
$$
where the space becomes $C(\Lambda_{\fa,\fb,\fc}^{d+1})$ if $p = \infty$, as in the definition in the Section 3.2, 
and we also define an analog of the $K$-functional by, for $r \in \NN$ and $\rho > 0$,  
$$
\Kb_r(f; \rho)_{p,\Wb_{\fa,\fb,\fc}^{\b,\g}} = 
  \inf_{g} \left\{ \|f - g\|_{L^p\big(\Wb_{\fa, \fb, \fc}^{\b,\g},\Lambda_{\fa, \fb, \fc}^{d+1}\big)} + 
  \rho^r \left \|\big (\fD_{\fa,\fb,\fc}^{\b,\g}\big )^\f r 2 g \right\|_{L^p\big(\Wb_{\fa, \fb, \fc}^{\b,\g},\Lambda_{\fa, \fb, \fc}^{d+1}\big)} \right \}, 
$$
where the infimum is taken over $g \in C^r(\Lambda_{\fa,\fb,\fc}^{d+1})$. Then the following analog of 
Theorem \ref{main-thmV0} holds. 

\begin{thm} \label{thm:approx-d}
Let $\b \ge 0$, $\g> -1$, and $f \in L^p(\Lambda_{\fa,\fb,\fc}^{d+1},\Wb_{\fa,\fb,\fc}^{\b,\g})$ if $1 \le p < \infty$, 
and $f \in C(\Lambda_{\fa,\fb,\fc})$ if $p = \infty$. Then, for $r \in \NN$ and $n =1,2,\ldots$, there hold 
\begin{enumerate} [\rm (i)]
\item {\it direct theorem}:
$$
    E_n(f)_{L^p\big(\Wb_{\fa, \fb, \fc}^{\b,\g},\Lambda_{\fa, \fb, \fc}^{d+1}\big)} 
        \le c \, K_r(f;  n^{-1})_{p, \Wb_{\fa, \fb, \fc}^{\b,\g}} ;
$$
\item {\it inverse theorem}:  
$$
   K_r(f; n^{-1})_{p, \Wb_{\fa, \fb, \fc}^{\b,\g}} 
         \le c \,n^{-r} \sum_{k=0}^n (k+1)^{r-1} E_k(f)_{L^p(\Wb_{\fa, \fb, \fc}^{\b,\g},\Lambda_{\fa, \fb, \fc}^{d+1})}.
$$
\end{enumerate}
\end{thm}

The definitions of the two quantities in the theorem are verbatim extensions of the one in Section 3.2, and 
so is the proof of the theorem. In the case of the circular cone in Example \ref{ex:cone}, this theorem is established
in \cite{X23a}, where the proof follows the general framework developed in \cite{X21} for spaces of homogeneous 
type that possess highly localized kernels, so that the main task in \cite{X23a} reduces to establishing highly 
localized kernels. The latter approach also applies on $\Lambda_{\fa,\fb,\fc}^{d+1}$ for $\Wb_{\fa,\fb,\fc}^{0, \g}$.

Let $\wh a \in C^\infty(\RR_+)$ be a non-negative function on the real line such that
$$
    \wh a(t) =1 \,\,  \hbox{if $0 \le t \le 1$}, \quad \hbox{and} \quad \wh a(t) = 0  \,\,\hbox{if $t \ge 2$}.
$$ 
For $X, Y \in \Lambda_{\fa,\fb,\fc}^{d+1}$ and $\g > -1$, define the kernel
$$
L_n \left( \Wb_{\fa,\fb,\fc}^{0, \g}; X, Y\right) = \sum_{k=0}^{2n} \wh a \left (\f k n \right) \Pb_k \left( \Wb_{\fa,\fb,\fc}^{0, \g}; X,Y\right).
$$
By \eqref{eq:ker=kerB} and the corresponding result on the unit ball \cite{PX}, this kernel is highly localized in the sense 
that it decays almost exponentially if $X$ and $Y$ are away from the diagonal. The precise statement requires the 
distance function on the domain $\Lambda_{\fa,\fb,\fc}^{d+1}$ defined by
$$
  \d_{\Lambda_{\fa,\fb,\fc}}(X,Y) = \arccos \left( \la X, Y\ra + \sqrt{1-\|\xb\|^2- \ft(\|\xb\|, t)^2}\sqrt{1-\|\yb\|^2- \ft(\|\yb\|, s)^2} \right) 
$$
for $X = (\xb,t), Y = (\yb,s)\in \Lambda_{\fa,\fb,\fc}^{d+1}$, which follows from the mapping \eqref{st-uv2} and the distance 
function $\d_\BB(X,Y) = \arccos \left( \la X, Y\ra + \sqrt{1-\|X\|^2}\sqrt{1-\|Y\|^2}\right)$  of $\BB^{d+1}$. The 
inequalities that quantify the localization follow from the corresponding inequalities on the unit ball by the change 
of variables, which we leave to the interested readers. For the case of the circular cone $\Lambda_{0,1,1}^{d+1}$, 
the inequalities were established by a fairly long estimate in \cite{X23a}, which can now be avoided by appealing to the 
corresponding inequalities for the unit ball. 

The $K$-functional defined by the spectral operator in Theorem \ref{thm:approx-d} has an equivalent modulus of 
smoothness, defined as a multiplier operator \cite[Definition 3.9]{X21}, which can also be migrated to our setting.
We shall not go in this direction since it involves another set of definitions and notations, and the result is essentially 
a mapping away from the results on the unit ball. Instead, we discuss an analog of Theorem \ref{thm:approx-d}
that uses another $K$-functional defined when $\b = 0$. 

To keep the notation simple, we consider the case $\fa =0$ and $\fb =1$ and denote 
$$
   \Lambda_\fc^{d+1} =  \Lambda_{0, 1, \fc}^{d+1} \quad \hbox{and}\quad \Wb_{\fc}^\g = \Wb_{0, 1,\fc}^{0,\g}. 
$$
With $\b = 0$, the weight function $\Wb_{\fc}^{\g}(\xb,t)$ is rotationally invariant in $\xb$. We introduce 
the angular derivatives $D_{i,j}$ defined by, for $\xb \in \RR^d$, 
$$
        D_{i,j} = x_i \partial_j - x_j \partial_i, \qquad 1 \le i,j \le d,
$$
which is the angular derivative in the polar coordinates of the $(x_i,x_j)$ plane. Let 
$$
  \phi_c(\xb,t) = \sqrt{1-(1-\fc) t^2 - \|\xb\|^2}
$$
for $(\xb, t) \in \Lambda_\fc^{d+1}$. We further define, as $\ft(\|\xb\|,t) = \sqrt{t^2- \fc \|\xb\|^2}$ when $\fa =0$ and $\fb =1$, 
$$
 \fD_i  = \partial_{x_i} + \fc \frac{x_i}{t} \partial_t, \quad 1 \le i\le d, \quad \hbox{and}
     \quad \fD_{d+1} =  \frac{\sqrt{t^2- \fc \|\xb\|^2}}{t}  \partial_t.
$$
It follows readily 
$$
  \fD_{i,j} : = x_i \fD_j - x_j \fD_i = x_i \partial_{x_j} - x_j \partial_{x_i} = D_{i,j}, \qquad 1 \le i, j \le d,
$$
and, furthermore, 
$$
   \fD_{i,d+1} := x_i \fD_{d+1} - \ft(\|\xb\|,t) \fD_i = -  \frac{\sqrt{t^2- \fc \|\xb\|^2}}{t} \left(t \partial_{x_i} - (1-c) x_i \partial_t \right), \quad 1 \le i \le d. 
$$

With these notations, we then define our second $K$-functional by 
\begin{align*}
 \wh K_r(f; \rho)_{p, \Wb_{\fc}^{\g}}  = \inf_{g \in C^r(\Lambda_\fc^{d+1})}
     \Big \{ \| f-g\|_{L^p(\Wb_{\fc}^\g, \Lambda_{\fc}^{d+1})} 
      &  + \max_{1\le i,j\le d+1} \left \| \fD_{i,j}^r g \right \|_{L^p(\Wb_{\fc}^\g, \Lambda_{\fc}^{d+1})} \\
  & + \max_{1\le i \le d+1} \left \| \phi_\fc^r \fD_i^r g \right\|_{L^p(\Wb_{\fc}^\g, \Lambda_{\fc}^{d+1})}   \Big \}.  
\end{align*}
This $K$ functional can also be used to give a characterization of the best approximation by polynomials. 

\begin{thm} \label{thm:approx-d-2}
Let $\g> -1$, and $f \in L^p(\Lambda_{\fc}^{d+1},\Wb_{\fc}^\g)$ if $1 \le p < \infty$, 
and $f \in C(\Lambda_{\fa,\fb,\fc})$ if $p = \infty$. Then, for $r \in \NN$ and $n =1,2,\ldots$, there hold 
\begin{enumerate} [\rm (i)]
\item {\it direct theorem}:
$$
    E_n(f)_{L^p(\Wb_{\fc}^{\g},\Lambda_{\fc}^{d+1})} 
        \le c \, \wh K_r(f;  n^{-1})_{p, \Wb_{\fc}^\g} +n^{-r} \|f\|_{L^p(\Wb_{\fc}^{\g},\Lambda_{\fc}^{d+1})};
$$
\item {\it inverse theorem}:  
$$
   \wh K_r(f; n^{-1})_{p, \Wb_{\fc}^{\g}} 
         \le c \,n^{-r} \sum_{k=0}^n (k+1)^{r-1} E_k(f)_{L^p(\Wb_{\fc}^{\g},\Lambda_{\fc}^{d+1})}.
$$
\end{enumerate}
\end{thm}

\begin{proof}
If we replace $\fD_i$ by $\partial_i$, $\phi_\fc$ by $\phi(\yb) = \sqrt{1-\|\yb|^2}$, and the norm by
the norm of $L^p(\Wb_\BB^\g, \BB^{d+1})$ with $\Wb(\yb) = (1-\|\yb\|^2)^\g$, then the $K$-functions 
become $\wh K_r(f; \rho)_{p, \Wb_\BB^\g}$ for functions defined on the unit ball $\BB^{d+1}$. Moreover, 
by \eqref{st-uv2}, it is easy to verify that $\fD_i g = \left(\partial_i g\circ \psi^{-1} \right) \circ \psi$, which is
in fact how $\fD_i$ is defined. Consequently, we can conclude that 
$
   \wh K_r(f, \rho)_{p,\Wb_{\fc}^\g} 
       = \wh K_r( f\circ \psi^{-1}, \rho)_{p,\Wb_\BB^\g},
$
so that both the direct and the inverse estimates follow from the corresponding results on the unit ball 
\cite[Theorem 6.6]{DaiX1}, using the same argument of Theorem \ref{main-thmV0}. 
\end{proof}

In particular, if $\fc =1$, then $\Lambda_\fc^{d+1}$ becomes the circular cone $\VV^{d+1}$, for which 
$\phi_\fc (\xb,t) = \sqrt{1-\|\xb\|^2}$. It should mention that there is a modulus of smoothness that is 
equivalent to $\wh K_r(f; \rho)_{p,\Wb_\BB^\g}$ on the unit ball, which is defined in terms of 
forward difference operators in the Euler angles and the Ditzian-Totik modulus of smoothness in
the $t$-variable. One can map this modulus of smoothness to the domain $\Lambda_\fc^{d+1}$ and 
preserve more or less the part in the Euler angles, but the portion on the $t$-variable becomes much
more involved. 

Finally, it is worth pointing out that if the domain $\VV^{d+1}$ is equipped with the Jacobi weight 
$\Wb(\xb,t) = (t^2-\|x\|^2)^{\b-\f12} (1-t)^\g$, then the orthogonal structure holds for all polynomials,
and there is no need to restrict to polynomials that are even in the last variable. Moreover, both
spectral operator and addition formula exist \cite{X20}, and it fits into the general framework developed
in \cite{X21}. However, the apex of the cone becomes a singular point in this setting, which impacts 
the behavior of the approximation \cite{GX}. Indeed, while the characterization via the spectral operator, 
as in Theorem \ref{main-thmV0}, remains hold, the analog of $\wh K_r(f; \rho)$ on $\VV^{d+1}$ differs 
fundamentally in the two settings.

\end{document}